\documentclass[pdflatex,sn-mathphys-ay]{sn-jnl}%

\usepackage{graphicx}%
\graphicspath{{images/}}%
\usepackage{multirow}%
\usepackage{amsmath,amssymb,amsfonts}%
\usepackage{amsthm}%
\usepackage{mathrsfs}%
\usepackage[title]{appendix}%
\usepackage{xcolor}%
\usepackage{textcomp}%
\usepackage{manyfoot}%
\usepackage{booktabs}%
\usepackage{algorithm}%
\usepackage{algorithmicx}%
\usepackage{algpseudocode}%
\usepackage{listings}%
\usepackage{wrapfig}%

\usepackage{tikz}%
\usepackage{float}%
\usepackage[justification=centering,font=small,labelfont=bf]{caption}%
\usepackage{hyperref}%
\usepackage{comment}%
\usepackage{tabularx}%
\allowdisplaybreaks
\usepackage{enumitem}

\theoremstyle{thmstyleone}%
\newtheorem{thm}{Theorem}%
\newtheorem{prop}[thm]{Proposition}%
\newtheorem{lem}[thm]{Lemma}%
\newtheorem{cor}[thm]{Corollary}%
\newtheorem{prob}[thm]{Problem}%

\theoremstyle{thmstylethree}%
\newtheorem{defn}[thm]{Definition}%
\newtheorem{assumptions}[thm]{Assumptions}

\theoremstyle{thmstyletwo}%
\newtheorem*{rmk}{Remark}%
\newtheorem*{sktch}{Sketch of Proof}%

\DeclareMathOperator{\R}{\mathbb{R}}

\DeclareMathAlphabet{\mymathbb}{U}{BOONDOX-ds}{m}{n}%
\DeclareMathOperator{\1}{\mymathbb{1}}

\DeclareMathOperator{\tr}{tr}

\DeclareMathOperator{\vol}{vol}
\DeclareMathOperator{\ad}{ad}

\newcommand{\lb}{\langle}
\newcommand{\rb}{\rangle}
\newcommand{\lp}{\left(}
\newcommand{\rp}{\right)}
\newcommand{\of}{\circ}
\newcommand{\T}{\mathcal{T}}

\newcommand{\Nu}{\text{Nu}}
\let\Im\relax
\newcommand{\Im}{\text{Im}}
\newcommand{\leb}{\mathscr{L}}
\newcommand{\diffomega}{\text{Diff}(\Omega)}
\newcommand{\sdiffomega}{\text{Sdiff}(\Omega)}
\newcommand{\probomega}{\text{Prob}(\Omega)}
\newcommand{\vecomega}{\text{vec}(\Omega)}

\newcommand{\grad}{\text{grad}}
\newcommand{\A}{\mathcal{A}}
\let\a\relax
\newcommand{\a}{\mathfrak{a}}
\let\d\relax
\newcommand{\d}{\mathfrak{d}}
\let\P\relax
\DeclareMathOperator{\P}{\mathcal{P}}
\newcommand{\s}{\mathfrak{s}}
\newcommand{\Reach}{\mathcal{R}}
\let\H\relax
\newcommand{\H}{\mathcal{H}}

\newcommand{\D}{\mathcal{D}}

\newcommand{\V}{\mathcal{V}}

\newcommand{\ph}{\mathord{\raisebox{0.15ex}{$\scriptstyle\bullet$}}}

\begin{document}

\title[Fluid Mixing and Incompressible Optimal Transport]{Fluid Mixing and Incompressible Optimal Transport}

\author*[1]{\fnm{Max} \sur{Emerick}}\email{memerick@ucsb.edu}

\author[2]{\fnm{John} \sur{Igraszek}}\email{jigraszek@ucsb.edu}

\author[1]{\fnm{Bassam} \sur{Bamieh}}\email{bamieh@ucsb.edu}

\affil*[1]{\orgdiv{Department of Mechanical Engineering}, \orgname{University of California}, \orgaddress{\city{Santa Barbara}, \postcode{93106}, \state{California}, \country{USA}}}

\affil[2]{\orgdiv{Department of Mathematics}, \orgname{University of California}, \orgaddress{\city{Santa Barbara}, \postcode{93106}, \state{California}, \country{USA}}}

\abstract{The problem of incompressible fluid mixing is long-studied, yet many questions remain open. This paper aims to address the question ``what do efficient flow fields for mixing look like, and how do they behave?'' Inspired by the dynamic and geometric approach to optimal mass transport, we formulate a version of the fluid mixing problem as an optimal control problem for transport of a passive scalar, in which the dynamics are given by the continuity equation together with an incompressibility constraint. We then develop a geometric framework around this problem which parallels that of classical optimal transport: the optimal control problem induces a metric on the set of reachable states (analogous to the Wasserstein metric), and this metric carries a formal Riemannian structure when the optimal control objective is induced by an inner product. Using this structure, we decompose our problem into a series of more fundamental geometric problems and derive (formal) necessary conditions for optimality, geodesic equations, and gradient flows for mixing. Specializing the control objective to the kinetic energy recovers the incompressible Euler equation, connecting our framework to the geometric hydrodynamics of Arnold. The resulting problem appears to be ill-posed, however, suggesting that the kinetic energy is an unnatural measure of mixing effort. Specializing instead to the enstrophy, which we argue is better suited to mixing, yields the Euler-Poincar\'e (EPDiff) geodesic equation for the right-invariant homogeneous $\dot{H}^1$ metric on the group of volume-preserving diffeomorphisms, and we present numerical simulations of both geodesics and gradient flows in this case. Along the way, we generalize the classical Helmholtz decomposition of vector fields, rederive the Clebsch representation of incompressible flows, and develop a parallel to the Otto calculus of optimal transport.}

\keywords{Optimal transport, 
fluid mixing, 
optimal control, 
diffeomorphism groups, 
passive scalar mixing, 
geometric hydrodynamics, 
gradient flows, 
incompressible Euler equation}

\pacs[MSC Classification]{49Q22 (Primary), 76F25, 58D05, 58B20, 93C20, 49K20 (Secondary)}

\maketitle

\begingroup
\insert\footins{\footnotesize A preliminary version of this work appeared in the Proceedings of the 64th IEEE Conference on Decision and Control \citep{Emerick2025}. That conference article was a preliminary announcement of the results and did not include the proofs, numerical simulations, or any of the material regarding the weighted Helmholtz decomposition, energy- or enstrophy-penalized cases, gradient flows, or the Lagrangian lift presented here.\par}
\endgroup

\tableofcontents

\section{Introduction}

The efficient mixing of fluids is an old problem, traceable at least back to the work of \citet{Boussinesq1877} and \citet{Reynolds1883}. It is both mathematically interesting and practically important in a wide array of applications, including industrial chemical production and food processing, fuel-air mixing in combustion engines, microfluidic mixing in lab-on-chip devices, and understanding the circulation of particulates in the ocean and atmosphere. As a result, this problem has been extensively studied over the years, and yet, there remain many important open questions.

In modern terms, the fluid mixing problem is usually posed as some variant of the following: given two initially unmixed fluids in some particular configuration, find an input/flow field/feedback control law which mixes the two fluids efficiently.

Early on, fluid mixing was mostly understood as it related to turbulence \citep{Boussinesq1877,Reynolds1883,Taylor1921,Kolmogorov1941}. The prevailing view was that ``turbulent flows mixed well, while laminar flows did not''. In the classical era of dynamical systems, this viewpoint was then slightly refined: it was not turbulence that mixed fluids per se, but the existence of chaotic attractors in the flow \citep{Ottino1989,Aref1984}. Since turbulent flows often exhibit chaotic behavior, they tend to mix well. However, laminar flows can also exhibit chaotic behavior, and this is extremely important for mixing which occurs on small scales, such as in microfluidic devices.

Later, more quantitative studies posed the problem as one of optimization: ``find the \emph{most} efficient input/flow field/control law/etc. for mixing two fluids''. This requires one to define -- in rigorous mathematical terms -- what it means to ``mix'' two fluids, and what it means for this process to be ``efficient''. Neither of these definitions are at all obvious, and many measures of both fluid mixedness and mixing efficiency have been proposed over the years \citep{Thiffeault2012}.

Further complicating this problem is the fact that the equations of fluid mechanics -- most famously the Navier-Stokes equations, but the Euler equations as well -- are not fully understood, and depend in a delicate way on various problem parameters, including dimensionality, boundary conditions, and the regularity of allowable flow fields. \citet{Bressan2003}, for example, conjectured that the maximum achievable rate of mixing depends on the regularity of allowable flow fields, and several results along these lines have been proven since the conjecture was first posed \citep{crippa2008estimates,Lunasin2012,Seis2013,Iyer2014,Zlatos2017,Alberti2019}.

On the engineering side, many methods for the control of mixing processes have been proposed as well. To mention just a few of the most relevant works, various groups have considered control for mixing by: switching between a set of prescribed velocity fields \citep{DAlessandro1999,Mathew2007,Gubanov2010,Gubanov2012}, destabilization \citep{Bamieh2001}, boundary feedback control \citep{Aamo2003,Hu2018}, optimal control \citep{Liu2008,Couchman2010,Eggl2020}, gradient descent \citep{Lin2011}, and reinforcement learning \citep{Konishi2022}. Since there is a large literature on the fluid mixing problem, it is very difficult to be exhaustive in our review, and we refer the reader to \citet{Aref2017} for further references.

One question which has not been fully answered in the literature (and to which this paper aims to contribute something) is: ``independent of parameterizations, what do efficient mixing fields look like, and how do they behave?'' The motivation for asking this question is that it allows one not only to develop control strategies, but to do \emph{codesign}. In other words, a better understanding of mixing fields informs the design of the mixing devices themselves -- not just the strategies used to control them.

In this paper, we develop an abstract approach to this problem which is inspired by the dynamic/geometric approach to optimal mass transport (following, e.g., \citealp{Benamou2000,Otto2001}). The key connection here is a reformulation of the mixing problem as an optimal control problem for transport of a passive scalar. In fact, the main problem that we study can essentially be considered ``dynamic optimal transport with an incompressibility constraint''. While we are not the first to investigate something like this (see, e.g., \citealp{Brenier1989,Brenier1999,Liu2019,Baradat2019}), the connection to the fluid mixing problem remains relatively undeveloped. (We note that one other very recent paper \citep{Hu2025} also makes this connection, but with a different problem formulation and an emphasis on computational rather than structural questions.) In this way, this paper sits at the confluence of several lines of work descending from \citet{Arnold1966} and including \citet{Ebin1970,marsden1983coadjoint,Brenier1989,Brenier1999,Otto2001} and more recently \citet{Brockett2007,Agrachev2009,khesin2013geometry,gallouet2018camassa,Agrachev2020,khesin2021geometric,Raginsky2024,Abdelgalil2025,khesin2025universal,bauer2026wasserstein} regarding various connections between the geometry and control of diffeomorphism groups, optimal transport, and fluid mechanics.

The rest of this paper proceeds as follows. We formulate (Section \ref{problem_formulation_sec}) a version of the fluid mixing problem as an optimal control problem of transporting a passive scalar, making contact with both the areas of classical fluid mixing and dynamic optimal transport. We show (Section \ref{problem_structure_sec}) that this optimal control problem induces a metric on the set of reachable states, analogous to the Wasserstein metric of classical optimal transport. This metric enables a decomposition of our problem into a series of more fundamental geometric problems. When the optimal control objective is induced by an inner product, this metric also carries the formal structure of a Riemannian metric, which affords us the apparatus to compute (formal) necessary conditions for optimality later. We interpret maximally efficient mixing fields as geodesics in this Riemannian space. We also provide a generalization of the classical Helmholtz decomposition called the \emph{weighted Helmholtz decomposition}, which helps us to interpret the geometry. Using the Riemannian structure, we derive necessary conditions for optimality/geodesic equations in the general case (Section \ref{general_k_sec}), and treat in detail both the energy-penalized (Section \ref{example_energy_sec}) and enstrophy-penalized (Section \ref{example_enstrophy_sec}) cases. In the energy-penalized case, the geodesic equations reduce to the incompressible Euler equation, making contact with both the classical theory of geometric hydrodynamics of \citet{Arnold1966} and the \emph{Otto calculus} of optimal transport \citep{Otto2001}. This derivation also recovers the classical Clebsch representation of an incompressible flow, with the notable feature that the Clebsch potentials are not posited as an ansatz, but emerge from the optimal control problem as the transported density and its costate. While making for a fascinating connection, the Euler equation is notoriously poorly behaved, ultimately meaning that the energy-penalized version of our problem is ill-posed. We discuss implications for fluid mixing before moving on to the enstrophy-penalized case, which is expected to be much better behaved. In the enstrophy-penalized case, the geodesic equations reduce to the Euler-Poincar\'e (EPDiff) geodesic equation for the right-invariant homogeneous $\dot{H}^1$ metric on the group of volume-preserving diffeomorphisms -- a limiting case of the Lagrangian-averaged Euler (LAE-$\alpha$) family of equations. As another application of the Riemannian structure, we investigate (Section \ref{gradient_flows_sec}) gradient flows, which provide a more computationally tractable, ``greedy'' method for control of mixing than the full optimal control solutions. We provide numerical simulations of both the geodesic equations and gradient flows in the enstrophy-penalized case. We then explain (Section \ref{lagrangian_lift_sec}) a unifying perspective on our results and the theories of optimal transport and geometric hydrodynamics in terms of diffeomorphism groups and group actions on homogeneous spaces. We finally conclude (Section \ref{conclusion_sec}) with a brief review and a discussion of the many possible future directions.

We emphasize at the outset that our results are formal and structural in nature, in the spirit of \citet{Otto2001}. We leave the state and input spaces, the control-effort norm, and the mixedness metric unspecified (see Section~\ref{problem_formulation_sec}), we are deliberately loose about the regularity of the density, the velocity field, and other key objects, and we do not address the well-posedness of any of the stated problems. We regard this as a feature of the approach: the structure which emerges then sheds light on which analytic assumptions may or may not be compatible with the framework.

\section{Problem Formulation} \label{problem_formulation_sec}

We formulate our version of the fluid mixing problem as an optimal control problem of transporting a passive scalar. Part of the challenge of the fluid mixing problem is that there is no clear most natural formulation, and so our approach is ``framework-first'', in that we try to make minimal assumptions and abstract as much as possible.

In the settings which we consider, we begin with two fluids, A and B, which are contained in some domain $\Omega$. $\Omega$ is assumed to be a connected and bounded set in $\R^n$ (or more generally, an $n$-dimensional Riemannian manifold, possibly with boundary). The $n = 2$ and $n=3$ cases are of course the most practically relevant, but we refrain from making any assumptions here for now.

We will model the fluids A and B as passive scalars $\rho_A$ and $\rho_B$, assumed to be functions from $\Omega \to \R$. We will interchangeably refer to $\rho_A$ and $\rho_B$ as \emph{passive scalars} or as \emph{densities} (as in ``probability densities'' or ``mass densities''). We emphasize, however, that $\rho_A$ and $\rho_B$ are \emph{passive}, that is, they do not participate in the underlying dynamics. In fact, one advantage of modeling the fluids this way is that it allows one to get \emph{away} from particular dynamic models such as Euler or Navier-Stokes. In the most classical formulation of the mixing problem, $\rho_A$ and $\rho_B$ might be indicator functions supported on complementary sets, but there are many reasons to want to relax this assumption. For notational purposes, we will use the symbol $\P$ to denote the space of allowable densities.

Our objective is to mix the two fluids, i.e., to bring the densities $\rho_A$ and $\rho_B$ ``close together'' in some appropriate sense. This ``closeness'' or ``mixedness'' is quantified by a metric $d$ which is defined on the space of densities $\P$. Thus, $A$ and $B$ will be well-mixed if $d(\rho_A,\rho_B)$ is small, and poorly mixed if it is large. Various metrics $d$ have been proposed in the literature, including various measures of entropy, the so-called ``mix-norm'', negative-order Sobolev norms, and the Wasserstein distance from optimal mass transport \citep{Thiffeault2012}. But again, we will refrain from making specific assumptions.

One assumption that we will make is that $\rho_A$ and $\rho_B$ fully determine each other (for example, if $\rho_B = 1 - \rho_A$). Then, instead of trying to control two densities to be close to each other, we may focus on just one density -- say, $\rho_A$ -- and try to control it towards some desired target $\rho_*$. In the classical mixing problem, $\rho_*$ might be the ``perfectly-mixed'' (i.e. uniform) density having the same total mass as $\rho_A$, however, there are many reasons to want to relax this assumption. Thus, in what follows, we will drop the subscript A and speak only of controlling the single density $\rho$.

Mixing in the real world typically occurs in two phases: first in an advection-dominated phase and then in a diffusion-dominated phase. It is this first advection-dominated phase which is primarily responsible for large-scale mixing (and is the process that we, as designers, have the most control over), while diffusion is primarily responsible for small-scale mixing and becomes significant only after large-scale mixing has already occurred. Thus, we will focus on the case of purely advective mixing in this paper. Here, the dynamics of the density $\rho$ are governed by the \emph{continuity equation}
\begin{equation} \label{continuity_eqn} 
	\partial_t \rho_t ~=~ - \nabla \cdot (\rho_t v_t) ,
\end{equation}
where $\partial_t$ denotes the time derivative, $t \in [0,T]$ the time variable, $\nabla \cdot$ the divergence operator, and $v$ the flow field/velocity field which transports the fluid. We emphasize that the subscript $t$ in $\rho_t$ and $v_t$ denotes a time \emph{index} rather than a time \emph{derivative}.

Often, the fluid is very nearly incompressible, which may be expressed by the additional \emph{incompressibility constraint}\footnote{\label{incompressibility_footnote}The zero-divergence constraint on $v$ is equivalent to preservation of volume under the flow (and thus incompressibility) by the divergence theorem.}
\begin{equation} \label{incompressibility_eqn}
	\nabla \cdot v_t \equiv 0 .
\end{equation}
In this case, the dynamics \eqref{continuity_eqn} may be equivalently written
\begin{equation}
    \partial_t \rho_t ~=~ - \nabla \rho_t \cdot v_t
\end{equation}
using the product rule $\nabla \cdot(\rho v) = \nabla \rho \cdot v + \rho \nabla \cdot v$ together with \eqref{incompressibility_eqn}. It is this constraint which gives fluid mixing much of its structure and associated challenges. In addition to incompressibility, allowable flow fields $v$ will also need to satisfy some appropriate boundary conditions, typically either the \emph{no-flux} condition $v \cdot \hat{n} \equiv 0$ on $\partial \Omega$ or the \emph{no-slip} condition $v \equiv 0$ on $\partial \Omega$, depending on the model considered. But again, we will postpone this choice, using the symbol $\V$ to denote the space of allowable vector fields. We will assume that $\V$ forms a vector space.

To define rigorously the notion of ``efficient'' mixing, we will introduce a norm $\| \cdot \|$ on the space of allowable vector fields $\V$. The norm $\| \cdot \|$ is taken as a proxy for ``mixing effort'' -- a large norm meaning an aggressive flow field and a small norm meaning a conservative one. An ``efficient'' flow field is thus one that makes $d(\rho,\rho_*)$ and $\| v \|$ small simultaneously. Various norms $\| \cdot \|$ have been proposed in the literature \citep{Thiffeault2012}, including the kinetic energy $\| \cdot \|_{L^2}$, viscous dissipation energy/enstrophy $\| D ( \cdot ) \|_{L^2} \simeq \| \cdot \|_{H^1}$, and ``palenstrophy'' $\| \Delta ( \cdot ) \|_{L^2} \simeq \| \cdot \|_{H^2}$. But again, we will refrain from making specific assumptions here.

Putting all of this together, we formulate our version of the fluid mixing problem as an optimal control problem which trades off between fluid mixedness and mixing effort as follows.

\begin{prob}[Optimal Mixing] \label{mixing_prob}
	Given initial and target densities $\rho_0,\rho_* \in \P$, time horizon $T$, norm $\| \cdot \| : \V \to \R$, metric $d: \P \times \P \to \R$, and trade-off parameter $\alpha >0$, solve
	\begin{equation} \label{mixing_prob_eqn}
		\inf_{\rho, v} \, \int_0^T \| v_t \|^2 \, dt \, + \, \alpha d^2(\rho_T,\rho_*) \quad \text{s.t.} \quad \begin{cases}
			\partial_t \rho = - \nabla \cdot (\rho v) \\
			\nabla \cdot v = 0 \\
			\rho_0 ~ \text{fixed}
		\end{cases} \hspace{-4mm}
	\end{equation}
	We refer to $\rho$ as the state, $\P$ as the state space, $v$ as the input, and $\V$ as the input space.
\end{prob}

A visual depiction of an example solution to this problem is shown in Figure~\ref{fig:mixing_schematic}.

\begin{figure}[ht]
    \centering
    \includegraphics[width=\linewidth]{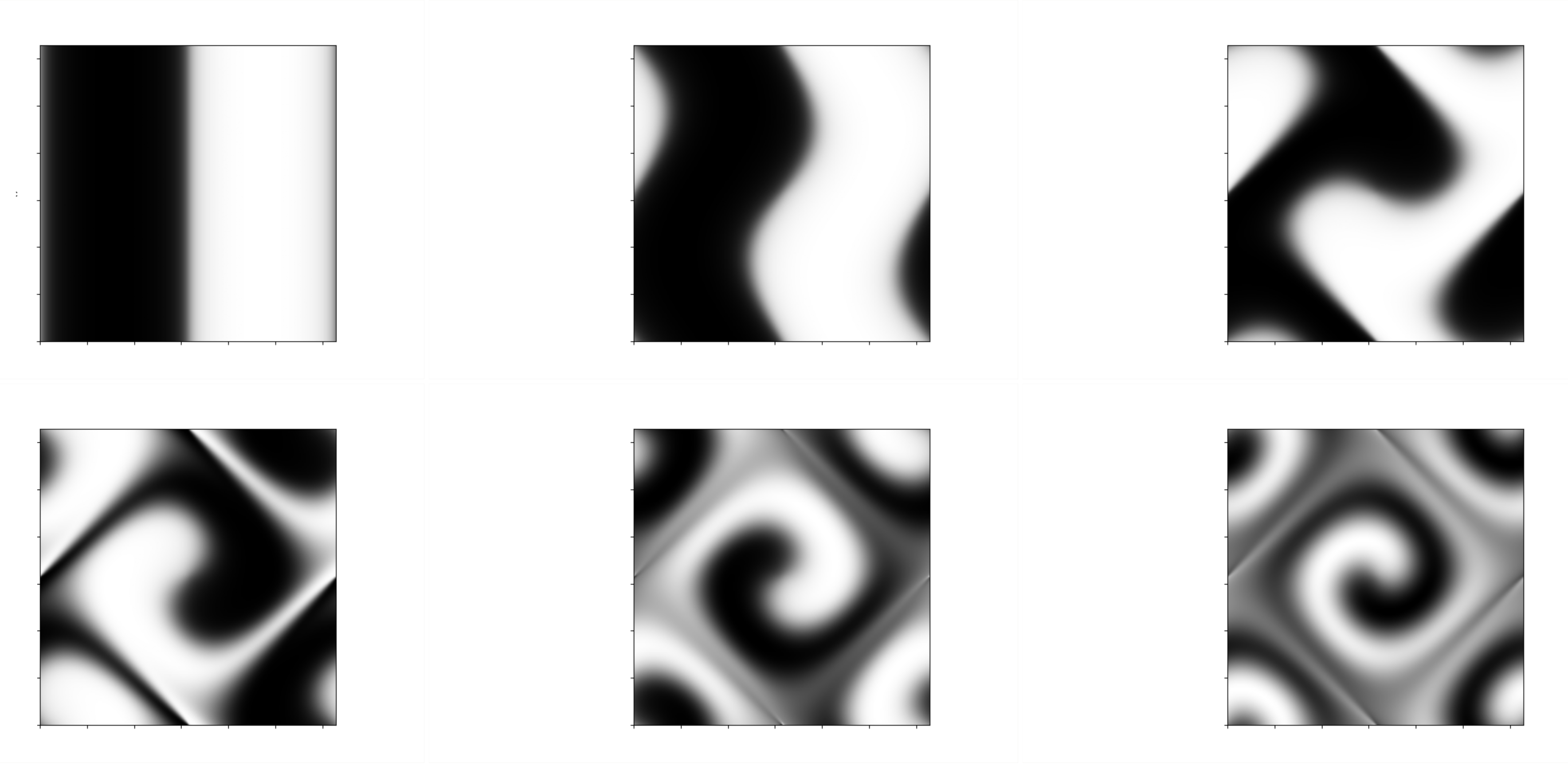}
    \caption{An example solution to Problem \ref{mixing_prob}. The two phases $\rho_A =: \rho$ and $\rho_B = 1 - \rho$ are initially unmixed (top left). A velocity field $v$ is applied and the two fluids begin to mix. Equivalently, $\rho$ begins to move towards the ``fully-mixed'' state $\rho_* = 0.5$, as measured in some appropriate mixedness metric $d$. The example here occurs over a two-dimensional domain with periodic boundary conditions.
    }
    \label{fig:mixing_schematic}
\end{figure}

This problem generalizes the classical fluid mixing problem, since by taking $\rho_0$ to be an indicator function on some set and $\rho_*$ to be the uniform density with the same total mass as $\rho_0$, we recover a classical formulation. However, the connection to optimal transport should also be clear, since by taking $\| v_t \| = \| v_t \|_{L^2(\rho_t)}$ and replacing the ``soft'' penalty $\alpha d^2(\rho_T,\rho_*)$ with the ``hard'' constraint $\rho_T = \rho_*$, we recover the dynamic formulation of optimal mass transport due to \citet{Benamou2000}, but with an additional incompressibility constraint $\nabla \cdot v = 0$.

There is a rich geometric structure in the above problem which we hope to effectively communicate in this paper. In short, this problem induces a metric on the set of reachable states, and this metric has a (formal) Riemannian structure if the norm $\| \cdot \|$ is induced by an inner product. This allows the above problem to be decomposed into a series of more fundamental geometric problems: computing geodesics, projections, critical points, and so on, as well as connecting it to the theory of geometric hydrodynamics, control of diffeomorphism groups, and the Otto calculus of optimal transport theory.

\section{Geometric Structure of Problem} \label{problem_structure_sec}

We now describe the geometric structure of Problem \ref{mixing_prob}. We begin (Section \ref{reachability_sec}) by describing the reachability structure of the problem, where it is shown that reachability is an equivalence relation on the state space $\P$. We construct (Section \ref{metrics_decomp_sec}) a metric on the reachable set which is induced by the optimal control problem, and use this metric to decompose Problem \ref{mixing_prob} into a series of geometric problems. We show (Section \ref{riemannian_structure_sec}) that this induced metric has the formal structure of a Riemannian metric when the norm $\| \cdot \|$ is induced by an inner product, which provides us with the apparatus needed to compute necessary conditions for optimality for our geometric problems later. Lastly, we introduce (Section \ref{helmholtz_sec}) a generalization of the classical Helmholtz decomposition called the \emph{weighted Helmholtz decomposition}, which provides additional interpretive apparatus for our results.

\subsection{Symmetries and Reachability} \label{reachability_sec}

The first question to be addressed is that of reachability: given two states $\rho_a$ and $\rho_b$, can we get from one to the other? We show that under the dynamics \eqref{continuity_eqn}-\eqref{incompressibility_eqn}, reachability is an equivalence relation on the state space $\P$. We also provide some basic necessary conditions for reachability. We begin with a definition.

\begin{defn}[Reachability] \label{reachability_def}
	Given two states $\rho_a, \rho_b \in \P$, we say that $\rho_b$ is \emph{reachable} from $\rho_a$ on the interval $[t_i,t_f]$ if there exists a trajectory $\rho: [t_i,t_f] \to \P$ and an input $v: [t_i,t_f] \to \V$ such that $\rho_{t_i} = \rho_a$, $\rho_{t_f} = \rho_b$, and the pair $(\rho,v)$ satisfy the dynamics \eqref{continuity_eqn}-\eqref{incompressibility_eqn} on the interval $[t_i,t_f]$ (possibly in some suitable weak sense).
\end{defn}

The system \eqref{continuity_eqn}-\eqref{incompressibility_eqn} possesses some strong symmetries, 
which will lead to similarly strong conclusions about the structure of the reachable set. The first symmetry is time-reparameterization.

\begin{lem}[Time-Reparameterization] \label{time_reparam_lem}
	Suppose that $(\rho,v)$ satisfy the dynamics \eqref{continuity_eqn}-\eqref{incompressibility_eqn} on $[t_i,t_f]$ with $\rho_{t_i} = \rho_a$ and $\rho_{t_f} = \rho_b$, and let $\sigma : [t_i',t_f'] \to [t_i,t_f]$ be any continuously differentiable and monotone nondecreasing function such that $\sigma(t_i') = t_i$ and $\sigma(t_f') = t_f$. Then the time-reparameterized functions $\tilde{\rho}_t := \rho_{\sigma(t)}$, $\tilde{v}_t := \dot{\sigma}(t) \, v_{\sigma(t)}$ satisfy the dynamics \eqref{continuity_eqn}-\eqref{incompressibility_eqn} on $[t_i',t_f']$, with $\tilde{\rho}_{t_i'} = \rho_a$ and $\tilde{\rho}_{t_f'} = \rho_b$.
\end{lem}

\begin{proof}
	Results directly from substituting $(\tilde{\rho},\tilde{v})$ into \eqref{continuity_eqn}-\eqref{incompressibility_eqn} and applying the chain rule.
\end{proof}

The above lemma immediately implies the following.

\begin{lem} \label{horizon_lem}
	If $\rho_b$ is reachable from $\rho_a$ on the finite time horizon $[t_i,t_f]$, $-\infty < t_i < t_f < \infty$, then it is reachable from $\rho_a$ on any other finite time horizon $[t_i',t_f']$.
\end{lem}

\begin{proof}
	By definition, there exists a solution $(\rho,v)$ to \eqref{continuity_eqn}-\eqref{incompressibility_eqn} with $\rho_{t_i} = \rho_a$ and $\rho_{t_f} = \rho_b$. Define $\sigma(t) := (t - t_i')(t_f - t_i)/(t_f'-t_i') + t_i$ and $(\tilde{\rho},\tilde{v})$ as in Lemma \ref{time_reparam_lem}. Then $(\tilde{\rho},\tilde{v})$ satisfy \eqref{continuity_eqn}-\eqref{incompressibility_eqn} with $\tilde{\rho}_{t_i'} = \rho_a$ and $\tilde{\rho}_{t_f'} = \rho_b$.
\end{proof}

This time-reparameterization symmetry of solutions will be a key tool in many of the arguments that follow. Also, note that this justifies dropping the dependence on time when discussing reachability.

In addition to time-reparameterization symmetry, the system \eqref{continuity_eqn}-\eqref{incompressibility_eqn} also possesses time-reversibility symmetry.

\begin{lem}[Time-Reversibility] \label{time_rev_lem}
	If $(\rho,v)$ satisfy the dynamics \eqref{continuity_eqn}-\eqref{incompressibility_eqn} on $[t_i,t_f]$ with $\rho_{t_i} = \rho_a$ and $\rho_{t_f} = \rho_b$, then the time-reversed functions $\tilde{\rho}_t := \rho_{t_f - t}$, $\tilde{v}_t := - v_{t_f - t}$ satisfy the dynamics \eqref{continuity_eqn}-\eqref{incompressibility_eqn} on $[t_i,t_f]$, with $\tilde{\rho}_{t_i} = \rho_b$ and $\tilde{\rho}_{t_f} = \rho_a$.
\end{lem}

\begin{proof}
	Results from direct substitution into \eqref{continuity_eqn}-\eqref{incompressibility_eqn}.
\end{proof}

We also define the following concatenation operation which will be useful in several of our arguments.

\begin{defn}[Concatenation]
	Given pairs $(\rho,v)$ and $(\rho',v')$ satisfying the dynamics \eqref{continuity_eqn}-\eqref{incompressibility_eqn} on the intervals $[t_i,t_f]$ and $[t_i',t_f']$ with $\rho_{t_f} = \rho'_{t_i'}$, we define the concatenation operation
	\begin{equation}
		\text{concat}((\rho,v),(\rho',v'))_t ~:=~
		\begin{cases}
			(\rho_t,v_t) &\text{if} \quad  t \in [t_i,t_f], \\
			(\rho'_{t - t_f + t_i'},v'_{t - t_f + t_i'}) &\text{if} \quad t \in (t_f, t_f + t_f' - t_i']
		\end{cases}
	\end{equation}
	It is immediate to check that $\text{concat}((\rho,v),(\rho',v'))$ satisfies the dynamics \eqref{continuity_eqn}-\eqref{incompressibility_eqn} on the interval $[t_i,t_f + t_f' - t_i']$.
\end{defn}

These properties are all used in the proof of the following.

\begin{prop} \label{equiv_rel_prop}
	Reachability is an equivalence relation on the state space $\P$. That is,
	\begin{enumerate}
		\item Every state is reachable from itself,
		\item $\rho_a$ is reachable from $\rho_b$ iff $\rho_b$ is reachable from $\rho_a$,
		\item If $\rho_b$ is reachable from $\rho_a$ and $\rho_c$ is reachable from $\rho_b$, then $\rho_c$ is reachable from $\rho_a$.
	\end{enumerate}
\end{prop}

\begin{proof}
	(1) Take $v \equiv 0$ (this is allowed because $\V$ is assumed to be a vector space). (2) Direct application of time-reversibility. (3) Direct application of concatenation.
\end{proof}

Recall that the equivalence relation partitions the state space $\P$ into equivalence classes of mutually reachable states. In light of this fact, we will use the notation $\Reach_{\rho_a}$ to denote the equivalence class of states which are reachable from $\rho_a$, and $\rho_a \sim \rho_b$ to mean that $\rho_a$ and $\rho_b$ are in the same equivalence class (i.e. are reachable from each other). When the initial density is fixed as in Problem \ref{mixing_prob}, we abbreviate the corresponding reachable set as $\Reach := \Reach_{\rho_0}$. Unlike in the classical optimal transport problem, the introduction of the incompressibility constraint \eqref{incompressibility_eqn} means that not all states are mutually reachable. The following problem is thus foundational.

\begin{prob}[Reachability]
	Given two states $\rho_a, \rho_b \in \P$, determine if $\rho_a \sim \rho_b$.
\end{prob}

Finding a general test to determine whether $\rho_a \sim \rho_b$ appears to be a genuinely nontrivial problem\footnote{This problem was posed to the attendees at the 2025 Optimal Transport Summer School at UC Santa Barbara. Thank you to all who participated in the resulting discussion, and Jethro Warnett in particular.}. That said, we do have some simple necessary conditions. For example, any conserved quantity gives rise to a necessary condition for reachability. The most fundamental conserved quantity for the system \eqref{continuity_eqn}-\eqref{incompressibility_eqn} is the total mass of the density $\int_\Omega \rho$, but we have some more discerning conserved quantities as well.

\begin{prop} \label{signature_prop}
	$\rho_\# \leb$ is a conserved quantity for the system \eqref{continuity_eqn}-\eqref{incompressibility_eqn}, where $\#$ denotes the (measure) pushforward, and $\leb$ the Lebesgue measure on $\Omega$.
\end{prop}

\begin{proof}
	As noted in Footnote~\ref{incompressibility_footnote}, the incompressibility constraint $\nabla \cdot v = 0$ implies that the flow generated by $v$ is volume-preserving. The level sets of $\rho$ are transported by this flow, so their volumes are preserved in time. The quantity $(\rho_\# \leb)(B) = \leb(\rho^{-1}(B)) = \vol(\rho^{-1}(B))$ is precisely the volume of the union of level sets $\{x \in \Omega ~:~ \rho(x) \in B\}$, so it too is invariant. Thus $\rho_\# \leb$ is conserved.
\end{proof}

We refer to the measure $\rho_\# \leb$ as the \emph{signature} of the density $\rho$.

\begin{cor}
	A necessary condition for $\rho_a \sim \rho_b$ is that they have the same signature: $(\rho_a)_\# \leb = (\rho_b)_\# \leb$.
\end{cor}

\begin{proof}
	Follows directly since $\rho_\# \leb$ is conserved.
\end{proof}

The intuition for the above condition is as follows. The signature $\rho_\# \leb$ is the pushforward of Lebesgue measure by the scalar function $\rho$ -- it is a measure on $\R$ which describes the relative volumes of the domain $\Omega$ which belong to level sets $\rho(\cdot) = r$. Under incompressibility, we cannot change the values that $\rho$ takes on, and thus we cannot change the levels themselves or the relative volumes of the level sets -- only move them around in the state space.

It is also worth pointing out that there is a dual characterization of the pushforward in terms of integration against test functions. That is,
\begin{equation}
    \int_\Omega f(\rho) \, d\leb ~=~ \int_{\R} f \, d (\rho_\# \leb)
\end{equation}
for all measurable $f: \R \to \R$. Thus, $\int_\Omega f(\rho)$ is also conserved for all measurable $f$. For example, all $L^p$ norms of $\rho$ are conserved.

\begin{rmk}
    The conservation of the signature has an immediate and important consequence for the mixing problem: the ``perfectly-mixed'' (uniform) target density is generally unreachable. This is why we require a ``soft'' final state penalty in place of a ``hard'' final state constraint in Problem \ref{mixing_prob}. Furthermore, since all $L^p$ norms are conserved, the distance $\| \rho_t - \rho_* \|_{L^p}$ is pinned to its initial value along any admissible trajectory. This is why mixedness must be measured in a weak sense (e.g. Wasserstein or negative-order Sobolev) rather than in $L^p$.
\end{rmk}

We also note that while the above condition is necessary for reachability, it is not in general sufficient. Sufficiency appears to depend delicately on the boundary conditions, on the regularity of allowed flow fields, and on the geometry of the level sets. For example, as we discuss in Section \ref{lagrangian_lift_sec}, if $v$ is required to be $C^k$, $k \geq 1$, then $\rho_0$ and $\rho_t$ can be related by a $C^k$ diffeomorphism. Thus, if $v$ is assumed to be sufficiently smooth, and if $\rho_a$ is unimodal whereas $\rho_b$ is bimodal, for example, then the two states cannot be mutually reachable.

We also point out that we treat only the case of finite-time, exact reachability here. It is entirely possible that two densities $\rho_a$, $\rho_b$ may not be mutually reachable in the sense of Definition \ref{reachability_def}, but that there may exist $(\rho,v)$ so that $\rho_t$ converges to $\rho_b$ starting from $\rho_a$ or vice versa. This asymptotic notion of reachability is a natural one to investigate given the weak-convergence picture of the above remark, but its treatment is left to future work.

\subsection{Induced Metrics and Problem Decomposition} \label{metrics_decomp_sec}

We show that Problem \ref{mixing_prob} can be broken down into three parts: 1) selection of a final density $\rho_f$, 2) selection of a path from $\rho_0$ to $\rho_f$, and 3) selection of a velocity field $v$ which generates that path. This is achieved by constructing a metric $m$ on the set of reachable states $\Reach \subset \P$ which is induced by the control effort. 
This approach is modeled on the dynamic formulation of optimal transport due to \citet{Benamou2000}.

While having a characterization of reachable states would be nice, practically speaking, this buys us nothing unless we are actually able to construct curves which connect them. We are thus brought to the following problem.

\begin{prob}[State-Transfer]
	Given initial and final states $\rho_0 \sim \rho_1$, construct a trajectory-input pair $(\rho,v)$ with endpoints $\rho_0$ and $\rho_1$ satisfying \eqref{continuity_eqn}-\eqref{incompressibility_eqn} on $[0,1]$.
\end{prob}

Note that by Lemma \ref{horizon_lem}, reachability is independent of the choice of time horizon, so we may fix it to $[0,1]$ without loss of generality. In general, the above problem seems to be a similarly challenging problem to that of reachability itself. One idea is that perhaps by imposing more structure on the problem, it might actually become easier to solve. For example, perhaps it is easier to find optimal trajectories than it is to find arbitrary ones.

\begin{prob}[Optimal State-Transfer]
	Given initial and final states $\rho_0 \sim \rho_1$ and a norm $\| \cdot \| : \V \to \R$, solve
	\begin{equation} \label{transfer_prob}
		\inf_{\rho, v} \, \int_0^1 \| v_t \|^2 \, dt \qquad \text{s.t.} \qquad (\mathrm{C1})~ \begin{cases}
			\partial_t \rho = - \nabla \cdot (\rho v) \\
			\nabla \cdot v = 0 \\
			\rho_0 ~ \text{fixed} \\
			\rho_1 ~ \text{fixed}
		\end{cases}
	\end{equation}
	We will refer to the value of this problem as the \emph{transfer cost} from $\rho_0$ to $\rho_1$ and denote it as $\T(\rho_0, \rho_1)$.
\end{prob}

We will mainly be interested in the case where the norm $\| \cdot \|$ arises from an inner product $\| v_t \|^2 = \lb v_t , v_t \rb$, as this will allow us to put a Riemannian structure on the reachable set $\Reach \subset \P$, as discussed in Section \ref{riemannian_structure_sec}. For now, it is enough to show that the above problem induces a metric 
on the reachable set. More precisely, we first define
\begin{equation} \label{metric_eqn}
	m(\rho_0,\rho_1) ~:=~ \inf_{\rho,v} \, \int_0^1 \| v_t \| \, dt \qquad \text{s.t.} \qquad (\mathrm{C1}) ,
\end{equation}
where $(\mathrm{C1})$ denotes the constraint set in \eqref{transfer_prob}. We show that $m$ forms a %
metric on the reachable set $\Reach$. We then show that the square root of the transfer cost $\T^{1/2}$ coincides with $m$. These properties follow directly from the symmetry properties of the system \eqref{continuity_eqn}-\eqref{incompressibility_eqn} and of the norm $\| \cdot \|$. Throughout, we will assume that the infimum in \eqref{metric_eqn} is attained, consistent with the formal nature of our results.

\begin{prop} \label{metric_prop}
	The quantity $m$ defined in \eqref{metric_eqn} forms a metric on the equivalence class of reachable states $\Reach$. That is,
	\begin{enumerate}
		\item $m(\rho_a, \rho_b) \geq 0$, with equality iff $\rho_a=\rho_b$,
		\item $m(\rho_a,\rho_b) = m(\rho_b,\rho_a)$,
		\item $m(\rho_a,\rho_b) + m(\rho_b,\rho_c) \geq m(\rho_a,\rho_c)$.
	\end{enumerate}
\end{prop}

\begin{proof}
	(1) Since the integrand $\| v_t \|$ is nonnegative, $m(\rho_a,\rho_b) \geq 0$. If $\rho_a = \rho_b$, then taking $v \equiv 0$ (so that $\rho \equiv \rho_a$ satisfies $(\mathrm{C1})$) gives $m(\rho_a,\rho_b) = 0$. Conversely, suppose $m(\rho_a,\rho_b) = 0$. By assumption, the infimum is attained, so there is an admissible pair $(\rho,v)$ with $l(v) = \int_0^1 \| v_t \| \, dt = 0$. This forces $\| v_t \| = 0$ for a.e.\ $t$, hence $v_t = 0$ for a.e.\ $t$, so that $\partial_t \rho = -\nabla\cdot(\rho v) = 0$ and $\rho$ is constant in time. Therefore $\rho_a = \rho_0 = \rho_1 = \rho_b$.
	
	(2) First, we define the quantity
	\begin{equation} \label{lv_orig_eq}
		l(v) ~:=~ \int_0^1 \| v_t \| \, dt ,
	\end{equation}
	so that $m = \inf \{ l(v) ~ \text{s.t.} ~ (\mathrm{C1}) \}$. By time-reversibility (Lemma \ref{time_rev_lem}) and homogeneity of $\| \cdot \|$, we see that if $\rho_b$ is reachable from $\rho_a$ with cost $l(v)$ using velocity field $v_t$, then $\rho_a$ is reachable from $\rho_b$ with cost $l(v)$ using velocity field $-v_{1-t}$. The result follows by passing to the infima.
	
	(3) First, we establish that the quantity $l(v)$ defined above is invariant under time-reparameterization. That is, let $\rho , v , \sigma, \tilde{\rho}, \tilde{v}$ be defined as in Lemma \ref{time_reparam_lem}. Then it holds by change-of-variables that
	\begin{multline}
		l(\tilde{v}) ~:=~ \int_{t_i'}^{t_f'} \| \tilde{v}_t \| \, dt ~:=~ \int_{t_i'}^{t_f'} \| \dot{\sigma}(t) \, v_{\sigma(t)} \| \, dt ~=~ \\
		\int_{t_i'}^{t_f'} \dot{\sigma}(t) \, \| v_{\sigma(t)} \| \, dt ~=~ \int_{t_i}^{t_f} \| v_\tau \| \, d\tau ~=:~ l(v) ,
	\end{multline}
	where the third equality uses $\dot{\sigma} \geq 0$ together with the homogeneity of $\| \cdot \|$, and the fourth is the change of variables $\tau = \sigma(t)$. This allows us to extend the definition of $l(v)$ to arbitrary time intervals. It is then immediate that $l$ is additive under concatenation, i.e., that $(\rho'',v'') = \text{concat}((\rho,v),(\rho',v'))$ implies $l(v'') = l(v) + l(v')$. 
	
	Now, fix $\varepsilon > 0$ and choose a velocity field $v$ taking $\rho_a$ to $\rho_b$ with $l(v) \leq m(\rho_a,\rho_b) + \varepsilon$ and a velocity field $v'$ taking $\rho_b$ to $\rho_c$ with $l(v') \leq m(\rho_b,\rho_c) + \varepsilon$. Their concatenation $v''$ takes $\rho_a$ to $\rho_c$. After reparameterizing to $[0,1]$ (which leaves $l$ unchanged by the established invariance), it is admissible in \eqref{metric_eqn}. Therefore
	\begin{equation}
		m(\rho_a,\rho_c) ~\leq~ l(v'') ~=~ l(v) + l(v') ~\leq~ m(\rho_a,\rho_b) + m(\rho_b,\rho_c) + 2\varepsilon .
	\end{equation}
	Letting $\varepsilon \to 0$ yields the result.
\end{proof}

Now, we show that the square root of the transfer cost coincides with $m$. Actually, we show a bit more.

\begin{lem} \label{p_action_lem}
	For all $T \in (0,\infty)$, $p \in [1,\infty)$, it holds that
	\begin{equation} \label{p_action_eq}
		m^p(\rho_a,\rho_b) ~=~ \inf_{\rho,v} ~ T^{p-1} \int_0^T \| v_t \|^p \, dt \qquad \text{s.t.} \qquad \begin{cases}
			\partial_t \rho = - \nabla \cdot (\rho v) \\
			\nabla \cdot v = 0 \\
			\rho_0 = \rho_a \\
			\rho_T = \rho_b
		\end{cases}
	\end{equation}
\end{lem}

\begin{proof}
	If $\rho_a = \rho_b$ then both sides vanish (taking $v \equiv 0$), so we may assume $\rho_a \neq \rho_b$ so that $l(v) > 0$. We further restrict attention to velocity fields with $\| v_t \| > 0$ for a.e.\ $t$. This is without loss of generality, since any rest interval may be removed by reparameterization without changing $l(v)$.
	
	Let $(\rho,v)$ be a pair satisfying the constraints in the right-hand side of Equation \eqref{p_action_eq}.
	
	(1) First, we show that the pair admits a reparameterization where $\| v_t \| = l(v) / T = \text{constant}$. Define
	\begin{equation}
		\sigma(t) ~:=~ \frac{T}{l(v)} \int_0^t \| v_s \| \, ds .
	\end{equation}
	Supposing $\| v_t \|$ is continuous, $\sigma$ is continuously differentiable, and, since $\| v_t \| > 0$ a.e., strictly increasing from $0$ to $T$. Therefore $\sigma$ admits an inverse $\sigma^{-1}$, and the inverse function theorem gives
	\begin{equation}
		\dot{(\sigma^{-1})}(t) ~=~ \frac{1}{\dot{\sigma}(\sigma^{-1}(t))} ~=~ \frac{l(v)}{T \, \| v_{\sigma^{-1}(t)} \|} \quad \text{a.e.}
	\end{equation}
	Define $(\tilde{\rho},\tilde{v})$ by reparameterization by $\sigma^{-1}$. Then $\tilde{v}_t := \dot{(\sigma^{-1})}(t) \, v_{\sigma^{-1}(t)}$ is a velocity field on $[0,T]$, and, using the homogeneity of $\| \cdot \|$ to factor out the positive scalar $\dot{(\sigma^{-1})}(t)$,
	\begin{multline}
		\left\| \tilde{v}_t \right\| ~=~ \left\| \dot{(\sigma^{-1})}(t) \, v_{\sigma^{-1}(t)} \right\| ~=~ \left\| \frac{l(v)}{T \, \| v_{\sigma^{-1}(t)} \|} \, v_{\sigma^{-1}(t)} \right\|  \\
		~=~ \frac{l(v)}{T} \left\| \frac{ v_{\sigma^{-1}(t)} }{\| v_{\sigma^{-1}(t)} \|} \right\|  ~=~ \frac{l(v)}{T} ~=~ \text{constant} \quad \text{a.e.},
	\end{multline}
	since the final vector has unit norm.
	
	(2) Now, we show that an optimal $v$ in the right-hand side of \eqref{p_action_eq} can be taken to be of this form. This follows from Jensen's inequality, as the convexity of $\| \cdot \|^p$, $p \in [1,\infty)$ implies that
	\begin{equation} \label{jensen_eq}
		\int_0^T \| v_t \|^p \, dt ~=~ T \int_0^T \| v_t \|^p \, \frac{dt}{T} ~\geq~ T \lp \int_0^T \| v_t \| \, \frac{dt}{T}  \rp^p ~=~ T \left( \frac{l(v)}{T} \right)^p ~=~ T^{1-p} \, l^p(v) ,
	\end{equation}
	with equality if $\| v_t \| = l(v) / T = \text{constant}$.
	
	(3) The result follows from scaling Equation \eqref{jensen_eq} by $T^{p-1}$ and passing to the infimum, recognizing that $m = \inf \{ l(v) ~ \text{s.t.} ~ (\mathrm{C1}) \}$ (see proof of Lemma \ref{metric_prop}).
\end{proof}

In particular, taking $T = 1$ and $p = 2$ above gives $m = \T^{1/2}$. In other words, the metric $m$ can be characterized as the square root of the transfer cost.

Now, observe that the transfer cost of the optimal state-transfer problem \eqref{transfer_prob} is exactly the same as the control energy portion of the cost for the mixing problem \eqref{mixing_prob_eqn} when $T = 1$ and the final state $\rho_1$ in \eqref{transfer_prob} is taken equal to $\rho_T$ in \eqref{mixing_prob_eqn}. Since the terminal cost $\alpha d^2(\rho_T,\rho_*)$ depends on the trajectory only through its endpoint $\rho_T$, we may break the minimization in \eqref{mixing_prob_eqn} into two parts: first over trajectories with fixed endpoint $\rho_f$ (necessarily reachable from $\rho_0$), and then over the endpoints $\rho_f$ themselves:
\begin{align}
	&~~\inf_{\rho, v} \int_0^T \| v_t \|^2 \, dt  +  \alpha d^2(\rho_T,\rho_*) \quad \text{s.t.} \quad \begin{cases}
		\partial_t \rho = - \nabla \cdot (\rho v) \\
		0 = \nabla \cdot v \\
		\rho_0 ~ \text{fixed}
	\end{cases}  \hspace{-7mm} \\
	~=~ &\inf_{\rho_f \in \Reach}\, \left[ \inf_{\rho, v} \, \int_0^T \| v_t \|^2 \, dt ~~ \text{s.t.} ~~ \begin{cases}
		\partial_t \rho = - \nabla \cdot (\rho v) \\
		0 = \nabla \cdot v \\
		\rho_0 ~ \text{fixed} \\
        \rho_T = \rho_f
	\end{cases}  \right]  \, + \,  \alpha d^2(\rho_f,\rho_*) \\
	~=~ &\inf_{\rho_f \in \Reach}\, \frac{1}{T} \, m^2(\rho_0,\rho_f)  \, + \, \alpha d^2(\rho_f,\rho_*) , \label{state_selection_eq}
\end{align}
where the final equality applies Lemma \ref{p_action_lem} with $p = 2$. With the above argument, we have just proven the following.
\begin{thm} \label{decomposition_thm}
	Any optimal solution to the mixing problem \eqref{mixing_prob_eqn} is given by an optimal solution to the state-transfer problem \eqref{transfer_prob} from $\rho_0$ to $\rho_f$ (linearly rescaled from the horizon $[0,1]$ to $[0,T]$), where $\rho_f$ solves the final state selection problem \eqref{state_selection_eq}.
\end{thm}
\begin{proof}
	Follows directly as outlined above.
\end{proof}

We have now broken the mixing problem down into two component subproblems: selection of a final state $\rho_f$ according to \eqref{state_selection_eq}, and the optimal state-transfer problem \eqref{transfer_prob} from $\rho_0$ to $\rho_f$. In principle, a suitably well-mixed final state $\rho_f$ can be found by solving \eqref{state_selection_eq} for a sufficiently large value of $\alpha$. In practice, this may require some computational ingenuity. Treatment of this point is left to future work.

\begin{rmk}
    This decomposition shows that the mixedness metric $d$ enters the solution of the mixing problem only through selection of the final state $\rho_f$. This cleanly separates the geometry of the reachable set $\Reach$ -- the subject of the remainder of this paper -- from the choice of $d$.
\end{rmk}

Now, we focus on solving the optimal state-transfer problem \eqref{transfer_prob}. We will show that this problem can be further broken down into selection of a path in the space of reachable densities $\rho: [0,1] \to \Reach \subset \P$ and selection of a velocity field $v$ which generates that path. The key step here is to use the \emph{metric derivative}.

\begin{defn}[Metric Derivative]
	Consider the space of reachable densities $\Reach \subset \P$ endowed with the metric $m$. Given an absolutely continuous curve $\rho:[0,1] \to (\Reach,m)$, we define its \emph{metric derivative}
	\begin{equation} \label{metric_deriv_def}
		|\dot{\rho}_t| ~:=~ \lim_{h\to 0} \frac{m(\rho_{t+h},\rho_t)}{|h|}  .
	\end{equation}
	The definition of absolute continuity ensures that the metric derivative is well-defined for a.e. $t \in [0,1]$.
\end{defn}

The metric derivative can be related back to the velocity field $v$ as follows.

\begin{lem} \label{metric_derivative_lemma}
	Let $\rho: [0,1] \to (\Reach,m)$ be absolutely continuous. Then its metric derivative satisfies for a.e. $t \in [0,1]$
	\begin{equation} \label{metric_derivative_lem}
		|\dot{\rho}_t| ~=~ \inf_{v_t} \| v_t \|  \quad \text{s.t.} \quad (\mathrm{C2}) ~ \begin{cases}
			\partial_t \rho_t = - \nabla \cdot (\rho_t v_t) \\
			0 = \nabla \cdot v_t
		\end{cases} 
	\end{equation}
\end{lem}

\begin{proof}
	The main idea of the proof is to recognize that the metric $m$ is induced by the following length structure:
	\begin{align}
		L(\rho) ~&:=~ \inf_v \, l(v) ~~ \text{s.t.} ~~ (\mathrm{C2}) ~ \forall t \\
		~&:=~ \inf_v \int_0^1 \| v_t \| \, dt ~~ \text{s.t.} ~~ (\mathrm{C2}) ~ \forall t \\
		~& \phantom{:} =~ \int_0^1 \left[ \inf_{v_t} \| v_t \| ~~ \text{s.t.} ~~ (\mathrm{C2}) \right] \, dt . \label{inf_v_eq}
	\end{align}
	That is, that
	\begin{align}
		m(\rho_0 , \rho_1) ~&:=~ \inf_{\rho,v} \, \int_0^1 \| v_t \| \, dt ~~ \text{s.t.} ~~ (\mathrm{C1}) \\
		~& \phantom{:} =~ \inf_\rho L(\rho) ~~ \text{s.t.} ~~ \rho_0, \rho_1 ~ \text{fixed} .
	\end{align}
	The fact that $L$ is a length structure follows from the fact that $l$ is a length structure, which follows from the properties of $l$ established in the proof of Proposition \ref{metric_prop}. The equality in \eqref{inf_v_eq} amounts to interchanging the infimum and the integral. Since the constraints $(\mathrm{C2})$ are pointwise-in-time, this interchange is valid provided a measurable near-optimal selection $t \mapsto v_t$ of the inner problem exists; we assume this throughout.
	
	With the metric $m$ thus defined, we can then define the \emph{intrinsic length} of an absolutely continuous curve $\rho$ by
	\begin{equation}
		L_m(\rho) ~:=~ \int_0^1 |\dot{\rho}_t| \, dt .
	\end{equation}
	It is a foundational result in metric geometry (see, e.g., \citet[Theorem 2.4.3]{Burago2001} and surrounding discussion) that if $L$ is lower-semicontinuous with respect to uniform convergence of $\rho$, then the intrinsic length $L_m$ coincides with the length $L$ defined via the length structure. We will assume this to hold, consistent with the formal nature of our results.
    Under this assumption, we have
	\begin{equation}
		\int_0^1 |\dot{\rho}_t| \, dt  ~=~ \int_0^1 \left[ \inf_{v_t} \| v_t \| ~~ \text{s.t.} ~~ (\mathrm{C2}) \right] \, dt
	\end{equation}
	for all absolutely continuous curves $\rho$, so that by the additivity of $L$, $L_m$ under concatenation, we must have
	\begin{equation}
		|\dot{\rho}_t| ~=~ \inf_{v_t} \| v_t \|  \quad \text{s.t.} \quad (\mathrm{C2})
	\end{equation}
	for a.e. $t \in [0,1]$.
\end{proof}

Note that in equations \eqref{metric_deriv_def} and \eqref{metric_derivative_lem}, the curve $\rho$ is considered to be fixed. In particular, $(\mathrm{C2})$ is considered as a set of static constraints for $v$ at each $t$ (depending on $\rho$) rather than as a dynamic constraint. We can then break the optimal state-transfer problem \eqref{transfer_prob} over paths $\rho$ as follows:
\begin{align}
	&\inf_{\rho, v} \, \int_0^1 \| v_t \|^2 \, dt \quad \text{s.t.} \quad (\mathrm{C1}) \\
	~=~ &\inf_{\rho}  \int_0^1 \left[ \inf_{v_t} \| v_t \|^2 ~~ \text{s.t.} ~~ (\mathrm{C2}) \right] \, dt
	\quad \text{s.t.} \quad \begin{cases}
		\rho \in AC([0,1]; \Reach) \\
		\rho_0 ~ \text{fixed} \\
		\rho_1 ~ \text{fixed}
	\end{cases} \\
	~=~ &\inf_{\rho}  \int_0^1 |\dot{\rho}_t|^2 \, dt
	\quad \text{s.t.} \quad \begin{cases}
		\rho \in AC([0,1]; \Reach) \\
		\rho_0 ~ \text{fixed} \\
		\rho_1 ~ \text{fixed}
	\end{cases} ,
\end{align}
where $AC([0,1];\Reach)$ denotes the set of absolutely continuous curves from $[0,1] \to \Reach$.

Recall from the proof of Lemma \ref{p_action_lem} that optimal paths may be taken to have constant $\| v_t \|$, thus constant metric derivative $|\dot{\rho}_t|$. Since $|\dot{\rho}_t|$ is constant, the Cauchy--Schwarz inequality on $[0,1]$ holds with equality, giving
\begin{equation}
	\int_0^1 |\dot{\rho}_t|^2 \, dt ~=~ \left( \int_0^1 |\dot{\rho}_t| \, dt \right)^2 ~=~ \text{length}^2(\rho) ,
\end{equation}
where $\text{length}(\rho) = L_m(\rho)$ is the intrinsic length of the curve $\rho$.

We have therefore proven the following.

\begin{thm} \label{state_transfer_geodesic_thm}
	Any optimal solution to the state-transfer problem \eqref{transfer_prob} is given by a pair $(\rho,v)$, where $\rho$ is a minimum-length absolutely continuous path in the reachable set $\Reach \subset \P$ with endpoints $\rho_0$ and $\rho_1$ and speed $|\dot{\rho}_t| \equiv m(\rho_0,\rho_1) = \text{constant}$, and $v$ minimizes at each time $t$ the left-hand side of \eqref{metric_derivative_lem}.
\end{thm}

\begin{proof}
	Follows directly as outlined above.
\end{proof}

\begin{rmk}
	The minimum-length curves connecting each $\rho_0$ to a given $\rho_1$ can also be characterized as steepest-descent curves (i.e. gradient flows) of the distance function $m(\cdot,\rho_1)$. If such gradients can be defined uniquely, this provides a feedback controller $k: \Reach \to \V$ achieving optimal state transfer. This is classical, and connects the metric space picture with optimal control theory and the Hamilton-Jacobi-Bellman equation. We will have more to say about this in Section \ref{gradient_flows_sec}.
\end{rmk}

We have now effectively decomposed our problem into three component subproblems: 1) selection of a final state $\rho_f$, 2) selection of a minimum-length continuous path which connects $\rho_0$ to $\rho_f$, and 3) selection of a minimum-norm velocity field $v$ which generates that path. A cartoon depiction of this decomposition is shown in Figure \ref{fig:prob_decomp}.

\begin{figure}[ht]
    \centering
    \includegraphics[width=0.5\linewidth]{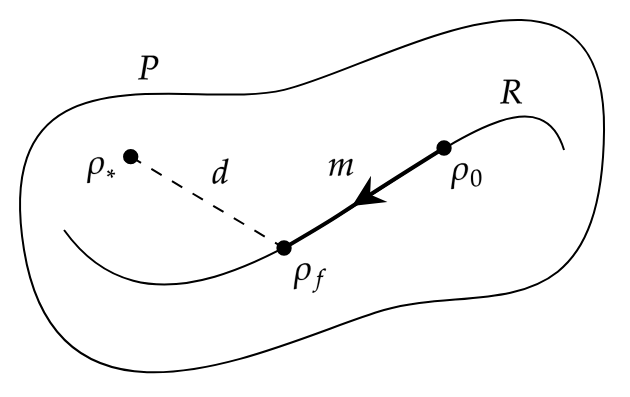}
    \caption{Depiction of problem decomposition as outlined in Theorems \ref{decomposition_thm} and \ref{state_transfer_geodesic_thm}. The solution to the mixing problem \eqref{mixing_prob} traverses a constant-speed geodesic from $\rho_0$ to $\rho_f$ in the reachable set $\Reach$. The final state $\rho_f$ is selected to optimally trade off between the distance $m$ on $\Reach$ induced by the mixing effort and the distance $d$ on $\P$ measuring mixedness. Note that the target density $\rho_*$ is generally not in the reachable set.}
    \label{fig:prob_decomp}
\end{figure}

While this picture recasts optimal state transfer as computation of minimum-length paths (i.e. geodesics), it is not yet constructive: given the metric $m$, we could in principle recover minimum-length paths, or given the minimum length paths, we could recover $m$, but at the moment we have direct access to neither. In the next subsection, we show that when the norm $\| \cdot \|$ arises from an inner product, $(\Reach,m)$ carries the formal structure of an infinite-dimensional Riemannian manifold. This provides us with the apparatus needed to compute necessary conditions for optimality and geodesic equations, a task which we take up in Section \ref{general_k_sec}.

\subsection{Riemannian Structure} \label{riemannian_structure_sec}

Now, we describe how the reachable set $(\Reach,m)$ carries the formal structure of an infinite-dimensional Riemannian manifold when the norm $\| \cdot \|$ is induced by an inner product. This section is modeled on Otto's work on the geometry of the 2-Wasserstein space of optimal transport theory \citep{Otto2001}.

To describe a Riemannian manifold, we need three things:
\begin{enumerate}
	\item a manifold $M$,
	\item a description of the tangent spaces $T_x M$, 
	\item an inner product $\lb \cdot , \cdot \rb_x$ on each tangent space.
\end{enumerate}
In our setting, we identify each of these with the following:
\begin{enumerate}
	\item the set of reachable states $\Reach \subset \P$,
	\item the set of perturbations $\tau$ at a given $\rho$ satisfying
	\begin{equation} \label{tangent_space_tau_eqn}
		\tau = - \nabla \cdot (\rho v) \quad \text{for some} \quad v \quad \text{s.t.} \quad \nabla \cdot v = 0 ,
	\end{equation}
    or alternatively, the corresponding equivalence classes of such $v$:
    \begin{equation} \label{equivalence_classes_eqn}
        [v]_\rho ~:=~\{ v ~:~ \tau = - \nabla \cdot (\rho v) , ~ \nabla \cdot v = 0 \} ,
    \end{equation}
	\item the inner product
    \begin{equation} \label{inner_prod_eqn}
        \lb \tau_1 , \tau_2 \rb_\rho ~:=~ \lb \bar{v}_1 , \bar{v}_2 \rb_\V ,
    \end{equation}
    where $\lb \cdot , \cdot \rb_\V$ is the inner product on $\V$ that is assumed to induce the norm $\| \cdot \|$, and $\bar{v}_i$ is the unique minimal-norm representative of the equivalence class \eqref{equivalence_classes_eqn} for $\tau_i$.
\end{enumerate}

We can see that this indeed recovers the metric as defined earlier, since by taking $\partial_t \rho = \tau_1 = \tau_2$ in \eqref{inner_prod_eqn}, we obtain exactly the characterization \eqref{metric_derivative_lem} of the metric derivative. The existence and uniqueness of the minimal norm representatives $\bar{v}_i$ for each equivalence class $[v_i]_\rho$ may be justified by the \emph{Hilbert projection theorem}, which proceeds as follows.

First, observe that each equivalence class is defined by two constraints: $\tau_i = - \nabla \cdot (\rho v_i)$ and $\nabla \cdot v_i = 0$. This latter constraint set is exactly the nullspace of the operator $\nabla \cdot \ph$, and is thus closed\footnote{We suppose that $\V$ is a Hilbert space, i.e., that it is complete in the norm $\| \cdot \|$, and that the operators $\nabla \cdot \ph$ and $\nabla \cdot (\rho \ph)$ are continuous as operators from $\V$ to an appropriate target. We verify these assumptions later in the cases we care about.}. The former constraint set is a translated nullspace: it is equal to $v_i + \Nu(\nabla \cdot (\rho \ph))$, for any $v_i$ which satisfies the constraint, and is thus also closed. The equivalence class $[v_i]_\rho$ is given by the intersection of these two constraint sets, which is then seen to be an affine set, and closed, since it is the intersection of closed sets. The Hilbert projection theorem then asserts the existence of a unique minimal-norm element $\bar{v}_i$ in this set. Furthermore, the operation taking each equivalence class $[v_i]_\rho$ to its minimal norm representative $\bar{v}_i$:
\begin{equation}
    \Theta_{\rho} : \V \to \V , ~ v_i \mapsto \bar{v}_i \in [v_i]_\rho
\end{equation}
is a $\V$-orthogonal projection.

We can in fact say more here. We can give a characterization of the minimum-norm element $\bar{v}_i$ as follows. Combining both constraint sets into the nullspace of a single block operator, we have
\begin{equation} \label{v_nullspace_eqn}
    [v_i]_\rho ~=~ v_i + \Nu \left( \begin{bmatrix}
        \nabla \cdot (\rho \ph) \\ \nabla \cdot \ph
    \end{bmatrix} \right) .
\end{equation}
The minimum norm element $\bar{v}_i$ must be perpendicular to the nullspace of the block operator above, and thus -- by the fundamental theorem of linear algebra -- an element of the closure of the image space of the adjoint:
\begin{equation} \label{v_imagespace_eqn}
    \bar{v}_i ~\in~ \overline{\Im \left( \begin{bmatrix}
        \nabla \cdot (\rho \ph) \\ \nabla \cdot \ph
    \end{bmatrix}^* \right)} ~=~ \overline{\Im \big( \begin{bmatrix}
        (\nabla \cdot (\rho \ph))^* & (\nabla \cdot \ph)^*
    \end{bmatrix} \big)} .
\end{equation}
A more explicit expression for this image space will be given in the next subsection, where it is assumed that $\| \cdot \|$ takes the particular form $\| K \cdot \|_{L^2}$. We make the formal assumption that this image is closed\footnote{This is valid if the image spaces of the two individual operators meet at a positive Friedrichs angle.} so that the closure may be dropped.

At this point, it would be very natural to take a route parallel to \citet{Otto2001}, and show that this Riemannian structure is induced by submersion on Arnold's group of volume-preserving diffeomorphisms \citep{Arnold1966}:
\begin{equation}
	\sdiffomega ~:=~ \{ \phi \in \diffomega ~:~  \det D \phi \equiv 1 \} .
\end{equation}
In this viewpoint, the nullspace and image space of the projection $\Theta_\rho$ may be understood as the \emph{vertical} and \emph{horizontal} components of this submersion. However, we will postpone this discussion to Section \ref{lagrangian_lift_sec} in the interest of using the machinery we have just developed to obtain constructive results regarding efficient mixing fields.

\subsection{The Weighted Helmholtz Decomposition} \label{helmholtz_sec}

We now give an explicit form to the minimal-norm $\bar{v}_i$'s described in the previous subsection using the \emph{weighted Helmholtz decomposition}. For concreteness of results, we narrow our focus to the case where $\V$ embeds in $L^2$ and where the norm on $\V$ takes the form $\| \cdot \| =: \| \cdot \|_\V = \| K \cdot \|_{L^2}$ for some linear operator $K: \V \to L^2$. Equivalently, $\| \cdot \| =: \| \cdot \|_\V$ is induced by the inner product
\begin{equation} \label{inner_product_eqn}
	\lb w , v \rb_\V ~:=~ \lb K w , K v \rb_{L^2} ~=~ \lb w , K^*K v \rb_{\V \hookrightarrow L^2} ,
\end{equation}
where $K^*$ denotes the $L^2$-adjoint of $K$. Here, it is essential to distinguish between the underlying \emph{set} of vector fields $\V$ and the \emph{geometry} placed upon it. $\V$ comes equipped with \emph{two} geometries: the geometry induced by $\lb \cdot , \cdot \rb_\V$ (defined via the operator $K$) and the geometry induced by the embedding $\V \hookrightarrow L^2$. Equation \eqref{inner_product_eqn} relates these two geometries to each other.

We require $K$ to be injective so that $\| K \cdot \|_{L^2}$ is a proper norm, densely defined (in $L^2$) so that $K^*$ is defined, and to have closed range so that (together with injectivity) $K^*K$ is boundedly invertible. This covers all of the cases that we care about, including the kinetic energy, enstrophy, and palenstrophy norms mentioned in Section \ref{problem_formulation_sec}.

With an explicit expression for the inner product $\lb \cdot , \cdot \rb_\V$, we may now compute the formal adjoint of the operator\footnote{Here $\mu \in \{ \rho , \1 \}$: the weight $\mu = \rho$ corresponds to the dynamic constraint $\tau = -\nabla \cdot (\rho v)$, and $\mu = \1$ to the incompressibility constraint $\nabla \cdot v = 0$.} $\nabla \cdot (\mu \ph)$ as follows. First, we compute the $L^2$-adjoint of this operator (i.e., treating $\V$ with the ambient $L^2$ inner product), using the divergence theorem to integrate by parts:
\begin{equation} \label{div_thm_eq}
    \int_\Omega \nabla \cdot (\phi \mu v) ~=~ \int_\Omega \nabla \phi \cdot (\mu v) + \phi \nabla \cdot (\mu v) ~=~ \int_\Omega \mu \nabla \phi \cdot v + \phi \nabla \cdot (\mu v) ~=~ \int_{\partial \Omega} \phi \mu v \cdot \hat{n} .
\end{equation}
By careful choice of domain, we can ensure that the boundary contribution is zero, thus
\begin{equation}
    \int_\Omega \mu \nabla \phi \cdot v ~=~ - \int_\Omega \phi \, \nabla \cdot (\mu v) \qquad \Leftrightarrow \qquad \langle \mu \nabla \phi , v \rangle_{L^2} ~=~ \langle \phi , - \nabla \cdot (\mu v) \rangle_{L^2} .
\end{equation}
This identifies $\mu \nabla \ph$ and $- \nabla \cdot (\mu \ph)$ as formal $L^2$-adjoints. To obtain the $\V$-adjoint, we may then apply \eqref{inner_product_eqn}, which gives the identity $\lb w , v \rb_{\V \hookrightarrow L^2} = \lb (K^*K)^{-1} w , v \rb_\V$, to obtain (noting that $K^*K$ and $(K^*K)^{-1}$ are self-adjoint):
\begin{equation}
    \langle \phi , - \nabla \cdot (\mu v) \rangle_{L^2} ~=~ \langle \mu \nabla \phi , v \rangle_{\V \hookrightarrow L^2} ~=~ \langle (K^*K)^{-1} \mu \nabla \phi , v \rangle_\V .
\end{equation}
We thus recognize $(K^*K)^{-1} \mu \nabla \ph$ as the formal $\V$-adjoint of $- \nabla \cdot (\mu \ph)$.

Thus, we expect that $\V$ admits the orthogonal decomposition
\begin{equation}
    \V ~\cong~ \Im\big( (K^*K)^{-1} \mu \nabla \ph \big) \oplus \Nu\big( -\nabla \cdot (\mu \ph) \big) .
\end{equation}
A depiction of this decomposition is shown in Figure \ref{fig:helmholtz}.

\begin{figure}[ht]
    \centering
    \includegraphics[width=0.8\linewidth]{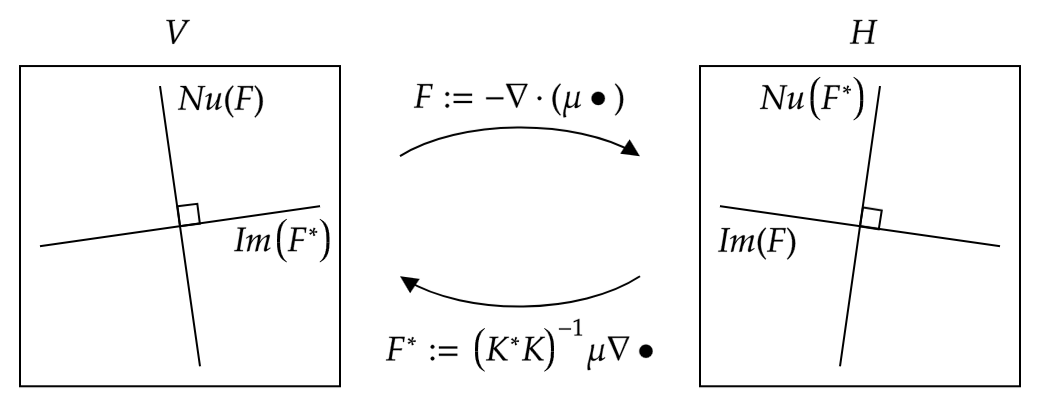}
    \caption{Depiction of the weighted Helmholtz decomposition as a fundamental relation between the nullspace of an operator and the image space of its adjoint. The operator $F:=-\nabla \cdot(\mu\ph)$ maps $\V \to \text{Im}(F) \subset \H$ while the adjoint $F^* = (K^*K)^{-1} \mu \nabla \ph$ maps $\H \to \Im(F^*) \subset \V$. The relevant subspaces are orthogonal complements, supposing that the respective image spaces are closed.}
    \label{fig:helmholtz}
\end{figure}

The weighted Helmholtz decomposition makes this formal decomposition rigorous. We provide a proof of this result under the following assumptions.

\begin{assumptions} \label{helmholtz_assumptions}
    \leavevmode
    \begin{itemize}
        \item The domain $\Omega$ is a compact, connected, $n$-dimensional Riemannian manifold with Lipschitz (or empty) boundary.
        \item The space $\V$ is a Hilbert space (i.e.\ complete in the norm $\| \cdot \| = \| K \cdot \|_{L^2}$).
        \item The operator $K$ is injective, densely defined in $L^2$, and has closed range.
        \item The density $\mu$ is measurable and uniformly bounded above and below by positive constants: $0 < c \leq \mu \leq C < \infty$.
        \item There exists a Hilbert space $\H$ of functions from $\Omega \to \R$ satisfying the hypotheses (i)--(iii) of the Lax-Milgram theorem stated in the proof below.
    \end{itemize}
\end{assumptions}

The final assumption is deliberately left abstract: $\H$ is the space of solutions to a certain Poisson problem, and depends both on the operator $K$ and on the relevant boundary conditions. We will give concrete examples of such $\H$ in the specific cases of the kinetic energy penalty (Section \ref{example_energy_sec}) and the enstrophy penalty (Section \ref{example_enstrophy_sec}). We now state the result.

\begin{thm}[Weighted Helmholtz Decomposition] \label{weighted_helmholtz_thm}
    Under Assumptions \ref{helmholtz_assumptions}, each $v \in \V$ admits a unique decomposition
	\begin{equation}
		v ~=~ (K^*K)^{-1} \mu \nabla \xi + w ,
	\end{equation}
	where $\xi \in \H$ and $\nabla \cdot (\mu w) = 0$, with the derivatives interpreted in the distributional sense. Furthermore, the decomposition is $\V$-orthogonal: $\lb (K^*K)^{-1} \mu \nabla \xi , w \rb_{\V} = \lb \mu \nabla \xi , w \rb_{L^2} = 0$.
\end{thm}

\begin{proof}
    The main idea of the proof is to recognize that the following two statements are equivalent:
    \begin{enumerate}[label=\arabic*)]
        \item $v = (K^*K)^{-1} \mu \nabla \xi + w$ for some $w$ with $\nabla \cdot (\mu w) = 0$;
        \item $\xi$ solves the Poisson problem $\nabla \cdot \big( \mu (K^*K)^{-1} \mu \nabla \xi \big) = \nabla \cdot (\mu v)$.
    \end{enumerate}
    Statement (2) follows from (1) by applying the operator $\nabla \cdot (\mu \ph)$ to both sides; conversely (1) follows from (2) by defining $w := v - (K^*K)^{-1} \mu \nabla \xi$. Both statements are interpreted in the distributional sense, i.e., in duality with test functions $\psi \in \H$.

    It remains to establish existence and uniqueness of a solution $\xi \in \H$ to the Poisson problem in (2). Written in weak form, this problem reads
    \begin{equation} \label{helmholtz_weak_form_eqn}
        a(\psi,\xi) ~:=~ \lb \mu \nabla \psi , (K^*K)^{-1} \mu \nabla \xi \rb_{\V \hookrightarrow L^2} ~=~ \lb \mu \nabla \psi , v \rb_{\V \hookrightarrow L^2} ~=:~ f(\psi) \qquad \text{for all } \psi \in \H .
    \end{equation}
    By the Lax-Milgram theorem, \eqref{helmholtz_weak_form_eqn} admits a unique solution $\xi \in \H$, depending continuously on $v$, provided that:
    \begin{enumerate}[label=\roman*)]
        \item Boundedness of the bilinear form: $|a(\psi,\xi)| \lesssim \| \psi \|_{\H} \| \xi \|_{\H}$;
        \item Coercivity of the quadratic form: $a(\xi,\xi) \gtrsim \| \xi \|_{\H}^2$;
        \item Boundedness of the linear functional: $|f(\psi)| \lesssim \| v \|_{\V} \| \psi \|_{\H}$.
    \end{enumerate}
    In the above, $A \lesssim B$ denotes that $A \leq c B$ for some positive constant $c$, and similarly for $\gtrsim$. These are precisely the hypotheses stated in the final item of Assumptions \ref{helmholtz_assumptions}.

    We note that hypothesis (iii) is precisely the statement that the operator $\nabla \cdot (\mu \ph)$ is continuous from $\V$ into the dual space $\H'$, since $f(\psi) = \lb \psi , - \nabla \cdot (\mu v) \rb_{L^2}$. This is the continuity required of the constraint operators in Section \ref{riemannian_structure_sec}, and it will be verified concretely along with (i)--(ii) in Sections \ref{example_energy_sec} and \ref{example_enstrophy_sec}.

    Finally, the decomposition is $\V$-orthogonal. Substituting $\psi = \xi$ into \eqref{helmholtz_weak_form_eqn} and using \eqref{inner_product_eqn} together with the self-adjointness of $(K^*K)^{-1}$, we obtain
    \begin{align}
        \lb \bar{v} , w \rb_\V ~&=~ \big\lb (K^*K)^{-1} \mu \nabla \xi , \, v - (K^*K)^{-1} \mu \nabla \xi \big\rb_\V \\
        ~&=~ \big\lb \mu \nabla \xi , \, v \big\rb_{\V \hookrightarrow L^2} ~-~ \big\lb \mu \nabla \xi , \, (K^*K)^{-1} \mu \nabla \xi \big\rb_{\V \hookrightarrow L^2} \\
        ~&=~ \lb \mu \nabla \xi , v \rb_{\V \hookrightarrow L^2} ~-~ a(\xi,\xi) ~=~ 0 . ~~ \qedhere
    \end{align}
\end{proof}

Some remarks concerning this theorem are now in order. First, the weighted Helmholtz decomposition generalizes the standard case where $\mu = \1$ is the uniform density and $K = I$ (so that $\V = L^2$). In that case, the theorem reads: every vector field $v \in L^2$ can be decomposed as $v = \nabla \xi + w$, with $\nabla \cdot w = 0$. The components $\nabla \xi$ and $w$ go by many names, including ``potential'', ``irrotational'', or ``conservative'' for $\nabla \xi$, and ``rotational'', ``solenoidal'', or ``incompressible'' for $w$.

In keeping with this usage, we will refer to $(K^*K)^{-1} \mu \nabla \xi$ and $w$ as the ``potential'' and ``rotational'' components of $v$, respectively -- even for general $\mu$, $K$, where $(K^*K)^{-1} \mu \nabla \xi$ may not be the gradient of some function and $w$ may not be divergence-free. We will also refer to the orthogonal projections
\begin{align}
    \Sigma_{\mu,K} &~:~ \V \to \Im\big( (K^*K)^{-1} \mu \nabla \ph \big) , ~~ v \mapsto (K^*K)^{-1} \mu \nabla \xi ,\\
    \Pi_{\mu,K} &~:~ \V \to \Nu\big( -\nabla \cdot (\mu \ph) \big) , ~~ v \mapsto w 
\end{align}
as the ``potential'' and ``rotational'' Helmholtz projections\footnote{The specific orthogonal projection $\Pi_{\1,I}$ is often referred to as the \emph{Leray} projection in the fluid mechanics literature.} with respect to $\mu$, $K$. We will often drop the explicit dependence on $\mu$, $K$ when it is clear from context. Since $\Sigma$, $\Pi$ are complementary orthogonal projections, recall also that they satisfy the identities
\begin{align}
    \Pi + \Sigma ~&=~ I , \label{projection_additive_eqn} \\
    \Pi \Sigma ~=~ \Sigma \Pi ~&=~ 0 . \label{projection_multiplicative_eqn}
\end{align}

Second, since one encounters many variations of the Helmholtz decomposition with various properties, we find it instructive to emphasize the following points regarding the version above. Here, we restrict our attention to the classical setting ($\mu = \1$, $K=I$):
\begin{itemize}
    \item The above decomposition is always unique and holds for all $v \in L^2$.
    \item Since general $L^2$ vector fields do not have well-defined boundary values, there are no boundary conditions imposed on $v$.
    \item Without boundary conditions, $v$ may in general contain harmonic components (i.e., components $u$ such that \emph{both} $u = \nabla \xi$ and $\nabla \cdot u = 0$). This freedom is then fixed by additional specification on $\xi$, that is, by specification of the space $\H$.
    \item If $\Omega$ has nonempty boundary, then this specification comes in the form of a boundary condition on $\xi$. For our application, we will place Dirichlet-type boundary conditions on $\xi$, which force the harmonic components into the rotational part.
    \item If $\Omega$ has empty boundary, then harmonic components \emph{always} go with the rotational part, since they are closed-but-not-exact forms. In this case, the free constant in $\xi$ is then fixed by the zero-mean condition $\int_\Omega \xi = 0$ rather than a boundary condition. Again, this enters in specification of the space $\H$.
    \item If $v$ has more regularity than mere $L^2$ so that its boundary values are well-defined and it also satisfies a no-flux condition $v \cdot \hat{n} = 0$ on $\partial \Omega$, then it has no harmonic components, and $\nabla \xi$ and $w$ also have well-defined boundary values and satisfy a no-flux condition.
\end{itemize}

Now, we may use this theorem to provide an explicit characterization of the minimum-norm $\bar{v}_i$'s of the previous subsection. Using the theorem, we verify
\begin{align}
    (\nabla \cdot (\rho \ph))^* ~&=~ (K^*K)^{-1} \rho \nabla \ph , \\
    (\nabla \cdot \ph)^* ~&=~ (K^*K)^{-1} \nabla \ph ,
\end{align}
and substituting into \eqref{v_imagespace_eqn} (with the assumed closure), obtain
\begin{equation} \label{vbar_potential_form_eqn}
    \bar{v}_i ~=~ (K^*K)^{-1} (\rho \nabla \lambda) + (K^*K)^{-1} (\nabla \gamma) ~=~ (K^*K)^{-1} (\rho \nabla \lambda + \nabla \gamma)
\end{equation}
for some $\lambda, \gamma: \Omega \to \R$. Since the minimal representative $\bar{v}_i$ must itself be incompressible, this expression can also be recast purely in terms of the Helmholtz projections. Applying $\Pi_{\1,K}$ to \eqref{vbar_potential_form_eqn} and noting that $\bar{v}_i = \Pi_{\1,K} \bar{v}_i$ and $\Pi_{\1,K} (K^*K)^{-1} \nabla \gamma = 0$ (since $(K^*K)^{-1} \nabla \gamma \in \Im \left( \Sigma_{\1,K} \right)$), we obtain
\begin{equation}
    \bar{v}_i ~=~ \Pi_{\1,K} \big( (K^*K)^{-1} \rho \nabla \lambda \big) \qquad \Leftrightarrow \qquad \bar{v}_i ~\in~ \Im \left( \Pi_{\1,K} \Sigma_{\rho,K}  \right) .
\end{equation}
These forms for $\bar{v}_i$ will appear throughout the sequel. We will give further interpretation to these expressions in the following section.

\section{Necessary Conditions and Geodesics: General $K$} \label{general_k_sec}

We now use the Riemannian structure described in the previous section to write (formal) necessary conditions and geodesic equations in the case where $\| \cdot \| =: \| \cdot \|_\V = \| K \cdot \|_{L^2}$. We retain the assumptions on $K$ from the previous section. These assumptions will be verified in the particular cases we care about in Sections \ref{example_energy_sec} and \ref{example_enstrophy_sec}.

First, we examine necessary conditions for optimality for the \emph{velocity field selection problem}, restated as follows for convenience:
\begin{equation} \label{vel_selection_eqn}
	\inf_{v} \, \| K v \|_{L^2}^2 \qquad \text{s.t.} \qquad \begin{cases}
		\tau = - \nabla \cdot (\rho v) \\
		0 = \nabla \cdot v
	\end{cases} .
\end{equation}
Recall that $\tau$ and $\rho$ are assumed to be given in the above problem. We assume that the given $\tau$, $\rho$ are such that the problem is feasible.

\begin{thm} \label{vel_sel_necc_cond_thm}
	For $v$ to be optimal in \eqref{vel_selection_eqn}, it is necessary that there exist $\lambda,\gamma : \Omega \to \R$ such that
	\begin{align}
		\tau &~=~ - \nabla \cdot (\rho v) \\
		0 &~=~ \nabla \cdot v \\
		0 &~=~ K^*K v - \rho \nabla \lambda - \nabla \gamma .
	\end{align}
\end{thm}

\begin{proof}
	We solve this problem by the method of Lagrange multipliers. We first form the Lagrangian
	\begin{equation}
		L(v,\lambda,\gamma) \, := \, \frac{1}{2} \langle K v , K v \rangle + \langle \lambda , \tau + \nabla \cdot (\rho v) \rangle + \langle \gamma , \nabla \cdot v \rangle .
	\end{equation}
	Note that we have scaled the objective in \eqref{vel_selection_eqn} by a factor of $1/2$ -- this simplifies the analysis but does not alter the resulting optimality conditions. We take variations with respect to each parameter, consider the linear parts, and set them equal to zero. This yields the equations
	\begin{align}
		0 ~&=~ \lb \tilde{\lambda} , \tau + \nabla \cdot (\rho v) \rb \\
		0 ~&=~ \lb \tilde{\gamma} , \nabla \cdot v \rb \\
		0 ~&=~ \lb K^*K v , \tilde{v} \rb + \lb \lambda , \nabla \cdot (\rho \tilde{v}) \rb + \lb \gamma , \nabla \cdot \tilde{v} \rb .
	\end{align}
	Since these equations must hold for all perturbations $\tilde{\lambda}, \tilde{\gamma}, \tilde{v}$, we can see that the right hand parts of the inner products in the first two equations must be equal to zero, and we recover our original constraints: $\tau = -\nabla \cdot (\rho v)$ and $\nabla \cdot v = 0$. In the third equation, we first use the divergence theorem to integrate by parts and then combine terms to obtain
	\begin{equation}
		0 ~=~ \lb K^*K v - \rho \nabla \lambda - \nabla \gamma , \tilde{v} \rb .
	\end{equation}
	By the same reasoning, since this equation must hold for all perturbations $\tilde{v}$, the left part of the inner product must be equal to zero, and thus we obtain the system
	\begin{align}
		\tau &~=~ - \nabla \cdot (\rho v) \\
		0 &~=~ \nabla \cdot v \\
		0 &~=~ K^*K v - \rho \nabla \lambda - \nabla \gamma . \label{last_eqn}
	\end{align}
\end{proof}

Since $K^*K$ is invertible by assumption, we may also rewrite Equation \eqref{last_eqn} as
\begin{equation} \label{v_form_eqn}
	v ~=~ (K^*K)^{-1}(\rho \nabla \lambda + \nabla \gamma) ~=~ (K^*K)^{-1}(\rho \nabla \lambda) + (K^*K)^{-1} \nabla \gamma .
\end{equation}
This is exactly the same characterization that we obtained for the optimal $\bar{v}$ using the weighted Helmholtz decomposition in Section \ref{helmholtz_sec}. However, we now have new interpretations for the scalar functions $\lambda$ and $\gamma$ as \emph{Lagrange multipliers} for the two constraints, which we can use to provide further intuition for this equation.

The multiplier $\lambda$ enforces the constraint $\tau = - \nabla \cdot (\rho v)$. Since this constraint describes the dynamics of the system, the interpretation of $\lambda$ is as a dynamic multiplier, or \emph{costate}, which enforces optimality. Indeed, the term $(K^*K)^{-1} (\rho \nabla \lambda) \in \Im(\Sigma_{\rho,K})$ is the minimal-norm velocity field $v$ such that $\tau = -\nabla \cdot (\rho v)$, and it \emph{would} be optimal for the velocity field selection problem \eqref{vel_selection_eqn} were it not for the incompressibility constraint $\nabla \cdot v = 0$.

The multiplier $\gamma$ then enforces this incompressibility constraint by projecting out the component of $(K^*K)^{-1}(\rho \nabla \lambda)$ which fails to meet incompressibility. This can be seen as follows. First, note that $v$ is incompressible, so $\Pi v = v$, and thus $\Sigma v = \Sigma \Pi v = 0$ by \eqref{projection_multiplicative_eqn}. Then, apply the projection $\Sigma = \Sigma_{\1,K}$ to \eqref{v_form_eqn} to obtain:
\begin{align}
    \Sigma v ~&=~ \Sigma \big( (K^*K)^{-1}(\rho \nabla \lambda) \big) + \Sigma \big( (K^*K)^{-1}\nabla \gamma \big) \\
    ~&=~ \Sigma \big( (K^*K)^{-1}(\rho \nabla \lambda) \big) + (K^*K)^{-1}\nabla \gamma \\
    ~&=~ 0 .
\end{align}
Rearranging, we have
\begin{equation}
    (K^*K)^{-1} \nabla \gamma ~=~ - \Sigma \big( (K^*K)^{-1}(\rho \nabla \lambda) \big) ,
\end{equation}
so that by \eqref{projection_additive_eqn}, we obtain
\begin{align}
    v ~&=~ (K^*K)^{-1} ( \rho \nabla \lambda) + (K^*K)^{-1} \nabla \gamma \\
    ~&=~ (K^*K)^{-1} ( \rho \nabla \lambda) - \Sigma \big( (K^*K)^{-1}(\rho \nabla \lambda) \big) \\
    ~&=~ (I - \Sigma) \big( (K^*K)^{-1} ( \rho \nabla \lambda) \big) \\
    ~&=~ \Pi \big( (K^*K)^{-1} ( \rho \nabla \lambda) \big) . \label{compact_v_eqn}
\end{align}
In this way, the role of $\gamma$ is analogous to the role of pressure in the incompressible Euler or Navier-Stokes equations.

It is also worth mentioning that the projection $\Pi = \Pi_{\1,K}$ can be written in the more explicit form
\begin{align}
    \Pi_{\1,K} ~&=~ I - (\nabla \cdot \ph)^\dagger  (\nabla \cdot \ph) \\
    ~&=~ I - (\nabla \cdot \ph)^* \big( \nabla \cdot (\nabla \cdot \ph)^* \big)^{-1} (\nabla \cdot \ph) \\
    ~&=~ I - (K^*K)^{-1} \nabla \big( \nabla \cdot (K^*K)^{-1} \nabla \big)^{-1} (\nabla \cdot \ph) ,
\end{align}
where $(\cdot)^\dagger$ denotes the pseudoinverse. This expression does not appear to simplify in general. (The notable exception is the case where $\nabla$, $K$ commute.)

With these interpretations in hand, we may now examine necessary conditions for optimality for the \emph{optimal state-transfer problem} \eqref{transfer_prob}.

\begin{thm} \label{state_transfer_necc_conds_thm}
	For the pair $(\rho,v)$ to be optimal in \eqref{transfer_prob} with $\| \cdot \| = \| K \cdot \|_{L^2}$, it is necessary that there exist $\lambda,\gamma : [0,1] \times \Omega \to \R$ such that $\rho,v,\lambda,\gamma$ satisfy
	\begin{align}
		\partial_t \rho &~=~ - \nabla \cdot (\rho v) \label{necc_cond_dyn_eqn} \\
		0 &~=~ \nabla \cdot v \\
		0 &~=~ \partial_t \lambda + \nabla \lambda \cdot v \\
		0 &~=~ K^*K v - \rho \nabla \lambda - \nabla \gamma  , \label{necc_cond_v_eqn}
	\end{align}
	in addition to the fixed boundary conditions on $\rho_0$ and $\rho_1$.
\end{thm}

\begin{proof}
	We solve this problem by the method of Lagrange multipliers just as in the proof of Theorem \ref{vel_sel_necc_cond_thm}. The Lagrangian is given by
	\begin{equation}
		L(\rho,v,\lambda,\gamma) \, := \, \frac{1}{2} \langle K v , K v \rangle + \langle \lambda , \partial_t \rho + \nabla \cdot (\rho v) \rangle + \langle \gamma , \nabla \cdot v \rangle .
	\end{equation}
	After taking variations and setting the linear parts equal to zero, we obtain
	\begin{align}
		0 ~&=~ \lb \tilde{\lambda} , \partial_t \rho + \nabla \cdot (\rho v) \rb \\
		0 ~&=~ \lb \tilde{\gamma} , \nabla \cdot v \rb \\
		0 ~&=~ \lb \lambda , \partial_t \tilde{\rho} + \nabla \cdot (\tilde{\rho} v) \rb \\
		0 ~&=~ \lb K^*K v , \tilde{v} \rb + \lb \lambda , \nabla \cdot (\rho \tilde{v}) \rb + \lb \gamma , \nabla \cdot \tilde{v} \rb .
	\end{align}
	The first, second, and fourth equations are identical to those investigated in the proof of Theorem \ref{vel_sel_necc_cond_thm} (with $\tau = \partial_t \rho$), and we obtain the same equations. For the third equation, we first split the inner product into terms involving $\partial_t \tilde{\rho}$ and $\nabla \cdot (\tilde{\rho} v)$, integrate the first in time (the boundary term vanishing since the endpoints $\rho_0, \rho_1$ are fixed, so that $\tilde{\rho}_0 = \tilde{\rho}_1 = 0$) and the second in space (again using the divergence theorem) to obtain
	\begin{equation}
		0 ~=~ \lb - \partial_t \lambda , \tilde{\rho} \rb + \lb -\nabla \lambda , \tilde{\rho} v \rb .
	\end{equation}
	Finally, we move $v$ into the left side of the second inner product and combine terms to obtain
	\begin{equation}
		0 ~=~ -\lb \partial_t \lambda + \nabla \lambda \cdot v , \tilde{\rho} \rb .
	\end{equation}
	By the same reasoning as before, the left side of the inner product must be equal to zero. Thus we obtain the system 
	\begin{align}
		\partial_t \rho &~=~ - \nabla \cdot (\rho v) \\
		0 &~=~ \nabla \cdot v \\
		0 &~=~ \partial_t \lambda + \nabla \lambda \cdot v \\
		0 &~=~ K^*K v - \rho \nabla \lambda - \nabla \gamma  . ~~ \qedhere
	\end{align}
\end{proof}

The above conditions are exactly the original dynamics (including incompressibility), the condition that $v_t$ be optimal for each $\partial_t \rho_t$ (as in Theorem \ref{vel_sel_necc_cond_thm}), and a new \emph{costate equation}, describing the evolution of the costate $\lambda$ that enforces optimality. Since optimal here means shortest path, the interpretation of the costate equation in the Riemannian context is as a geodesic equation\footnote{In Riemannian geometry, the geodesic equation describes curves of zero intrinsic acceleration, and can also be derived via the Riemannian curvature. Here, these curves are derived via a variational principle instead.}.

Thus, on one hand, these necessary conditions describe geodesics/shortest paths in this Riemannian framework. On the other hand, since we have constructed the metric $m$ from the transfer cost in the optimal state-transfer problem, these geodesics exactly describe flow fields which transfer states with minimal effort (i.e. with maximal efficiency). Therefore, maximally efficient mixing fields must always be geodesics of this form.

Rewriting the dynamics as $\partial_t \rho = - \nabla \rho \cdot v$ (equivalent under incompressibility), and repackaging the constraints on $v$ as in \eqref{compact_v_eqn}, we may write the system \eqref{necc_cond_dyn_eqn}-\eqref{necc_cond_v_eqn} in the following compact form:
\begin{align}
	\partial_t \rho ~&=~ - \nabla \rho \cdot \Pi_{\1,K} \big( (K^*K)^{-1} (\rho \nabla \lambda) \big) \label{gen_k_dynamics_1} \\
	\partial_t \lambda ~&=~ - \nabla \lambda \cdot \Pi_{\1,K} \big( (K^*K)^{-1} (\rho \nabla \lambda) \big) . \label{gen_k_dynamics_2}
\end{align}
Now, this system has a simple structure: $\rho$ and $\lambda$ are transported together by the velocity field $v = \Pi_{\1,K} \big( (K^*K)^{-1} (\rho \nabla \lambda) \big)$. Thus, it should be possible to write a single PDE for the evolution of $\rho \nabla \lambda$, and thus for the evolution of $v$. And indeed, we can.

\begin{thm}
	The evolution of the velocity field $v = \Pi_{\1,K} \big( (K^*K)^{-1} (\rho \nabla \lambda) \big)$ under the dynamics \eqref{gen_k_dynamics_1}-\eqref{gen_k_dynamics_2} satisfies the equation
	\begin{equation}
		\partial_t v ~=~ - \Pi_{\1,K} \big( (K^*K)^{-1} \big( 2 \, \operatorname{skew} (D(K^*Kv))[v] \big) \big) .
	\end{equation}
\end{thm}

\begin{proof}
	Differentiate $v = \Pi_{\1,K} \big( (K^*K)^{-1} (\rho \nabla \lambda) \big)$ with respect to time to obtain
	\begin{align}
		\partial_t v ~&=~ \Pi_{\1,K} (K^*K)^{-1} \left[ (\partial_t \rho) \nabla \lambda + \rho \nabla (\partial_t \lambda) \right] \\
		~&=~ -\Pi_{\1,K} (K^*K)^{-1} \left[ (\nabla \rho \cdot v) \nabla \lambda + \rho \nabla (\nabla \lambda \cdot v) \right] .
	\end{align}
	Since $\Pi_{\1,K}$ annihilates any terms of the form $(K^*K)^{-1}(\nabla \gamma)$, the operator $\Pi_{\1,K} (K^*K)^{-1}$ annihilates pure gradients, and we obtain the integration-by-parts-like identity
	\begin{equation}
		\Pi_{\1,K} (K^*K)^{-1} \left[ \nabla (fg) \right] ~=~ \Pi_{\1,K} (K^*K)^{-1} \left[ f \nabla g + g \nabla f \right] ~=~ 0 ,
	\end{equation}
	and hence $\Pi_{\1,K} (K^*K)^{-1} [ f \nabla g ] = \Pi_{\1,K} (K^*K)^{-1} [ - g \nabla f ]$. Using this to move the derivative from $\nabla \lambda \cdot v$ onto $\rho$ in the right term, and then applying the tensor identities,
	\begin{align}
		\partial_t v ~&=~ -\Pi_{\1,K} (K^*K)^{-1} \left[ (\nabla \rho \cdot v) \nabla \lambda - \nabla \rho (\nabla \lambda \cdot v) \right] \\
		~&=~ -\Pi_{\1,K} (K^*K)^{-1} \left[ (\nabla \lambda \otimes \nabla \rho) [v] - (\nabla \rho \otimes \nabla \lambda) [v] \right] \\
		~&=~ -\Pi_{\1,K} (K^*K)^{-1} \left[ (\nabla \lambda \otimes \nabla \rho - \nabla \rho \otimes \nabla \lambda) [v]  \right] \\
		~&=~ -\Pi_{\1,K} (K^*K)^{-1} \left[ 2 \, \operatorname{skew} (\nabla \lambda \otimes \nabla \rho) [v]  \right] .
	\end{align}
	Now, starting from the expression $K^*Kv = \rho \nabla \lambda + \nabla \gamma$, differentiate in space
	\begin{equation}
		D(K^*Kv) ~=~ D(\rho \nabla \lambda + \nabla \gamma) ~=~ \nabla \lambda \otimes \nabla \rho + \rho \nabla^2 \lambda + \nabla^2 \gamma ,
	\end{equation}
	and observe that $\operatorname{skew}(\nabla \lambda \otimes \nabla \rho ) = \operatorname{skew}(D(K^*Kv))$, since the latter two terms are symmetric. Substitute this into the above expression for $\partial_t v$ to obtain
	\begin{equation}
		\partial_t v ~=~ -\Pi_{\1,K} \big( (K^*K)^{-1} \big( 2 \, \operatorname{skew} (D(K^*Kv))[v] \big) \big) . \qedhere
	\end{equation}
\end{proof}

Thus, we can equivalently write the coupled system \eqref{gen_k_dynamics_1}-\eqref{gen_k_dynamics_2} as the one-way driven system
\begin{align}
    \partial_t \rho ~&=~ - \nabla \rho \cdot v \\
    \partial_t v ~&=~ - \Pi_{\1,K} \big( (K^*K)^{-1} \big( 2 \, \operatorname{skew} (D(K^*Kv))[v] \big) \big) . \label{v_dynamics_eqn}
\end{align}
The interpretation of the system \eqref{gen_k_dynamics_1}-\eqref{gen_k_dynamics_2} as geodesic equations in the Riemannian setting has a counterpart here as well -- the equation \eqref{v_dynamics_eqn} governing the evolution of $v$ may be interpreted as a geodesic equation in the space of flow maps. It is a classic result of Riemannian geometry that \emph{horizontal} geodesics on this larger space push down to geodesics in the quotient under the Riemannian submersion. We will say more about this in Section \ref{lagrangian_lift_sec}.

\begin{rmk}
	In terms of the \emph{momentum} $\mathfrak{m} := K^*K v$, the skew term may be expanded as $2 \, \operatorname{skew}(D\mathfrak{m})[v] = (v \cdot \nabla) \mathfrak{m} + (Dv)^\top[\mathfrak{m}] - \nabla(\mathfrak{m} \cdot v)$, where the pure-gradient term is annihilated by the outer projection. The evolution equation \eqref{v_dynamics_eqn} thus reads
	\begin{equation}
		\partial_t v ~=~ - \Pi_{\1,K} \big( (K^*K)^{-1} \big( (v \cdot \nabla) \mathfrak{m} + (Dv)^\top[\mathfrak{m}] \big) \big) , \qquad \mathfrak{m} = K^*Kv .
	\end{equation}
	Here $K^*K$ acts as an \emph{inertia operator} relating velocity to momentum, and the above expression is the form taken by the Euler-Poincar\'e (EPDiff) geodesic equations on a group of diffeomorphisms. We develop this connection in Section \ref{lagrangian_lift_sec}.
\end{rmk}

For now, we investigate two concrete special cases of the above theory most relevant to applications.

\section{Special Case: Energy-Penalized Flow Fields} \label{example_energy_sec}

We now specialize the general theory of Section \ref{general_k_sec} to the case where $v$ is penalized based on its kinetic energy, i.e., where $\| \cdot \| = \| \cdot \|_{L^2}$, or $K = I$ (so that $\V = L^2$). First, we verify our assumptions on $\V$ and $K$ from the previous sections.

We retain our prior assumptions on $\Omega$ and $\mu \in \{ \rho, \1 \}$: $\Omega$ is a compact, connected, $n$-dimensional Riemannian manifold with Lipschitz (or empty) boundary, and $\rho$ is measurable and uniformly bounded above and below by positive constants: $0 < c \leq \rho \leq C < \infty$.

The space $\H$ of scalar potentials is taken to be (scalar-valued) $H^1$ -- the set of $L^2$ scalar functions with $L^2$ weak gradient -- with either homogeneous Dirichlet boundary conditions (if $\Omega$ has boundary) or the zero-mean condition $\int_\Omega \xi = 0$ (if $\Omega$ is boundaryless). We refer to these spaces collectively as ``homogeneous $H^1$'' or $\dot{H}^1$.

\begin{prop}
    Assumptions \ref{helmholtz_assumptions} hold under the choices: $\V = L^2$, $K=I$, and $\H = \dot{H}^1$.
\end{prop}

\begin{proof}
    Since $\V = L^2$ and $K=I$, several of these assumptions are trivial: $\V$ is a Hilbert space, and $K$ is injective, densely defined, and has closed range. All that remains is to verify the hypotheses of Lax-Milgram for the choice of $\H$.

    Whether $\Omega$ has boundary or not, the choice of $\dot{H}^1$ ensures that the Poincar\'e inequality $\| \xi \|_{L^2} \lesssim \| \nabla \xi \|_{L^2}$ holds, implying that the norms $\| \cdot \|_{H^1} := (\| \cdot \|_{L^2}^2 + \| \nabla \cdot \|_{L^2}^2)^{1/2}$ and $\| \cdot \|_{\dot{H}^1} := \| \nabla \cdot \|_{L^2}$ are equivalent. The hypotheses of Lax-Milgram are thus verified for $\mu \in \{ \rho , \1 \}$:
    \begin{align}
        | a(\psi,\xi)| ~&=~ | \lb \mu \nabla \psi , \mu \nabla \xi \rb_{L^2} | ~\leq~ \| \mu \nabla \psi \|_{L^2} \| \mu \nabla \xi \|_{L^2} ~\lesssim~ \| \psi \|_{\dot{H}^1} \| \xi \|_{\dot{H}^1} ; \\
        a(\xi, \xi) ~&=~ \| \mu \nabla \xi \|_{L^2}^2 ~\gtrsim~ \| \xi \|_{\dot{H}^1}^2 ;\\
        | f(\psi) | ~&=~ | \lb \mu \nabla \psi , v \rb_{L^2} | ~\leq~ \| \mu \nabla \psi \|_{L^2} \| v \|_{L^2} ~\lesssim~ \| v \|_{L^2} \| \psi \|_{\dot{H}^1} ,
    \end{align}
    where the bounds marked $\lesssim$ use only the uniform bounds $0 < c \leq \mu \leq C < \infty$. In particular, hypothesis (iii) identifies $\nabla \cdot (\mu \ph)$ as continuous from $L^2$ into $\dot{H}^{-1}$, the dual of $\dot{H}^1$.
\end{proof}

With these assumptions verified, we now proceed to specialize the necessary conditions from the previous section. The weighted Helmholtz projection $\Pi = \Pi_{\1,I}$ reduces to the classical \emph{Leray} projection in this setting. Since $K^*K = I$, the necessary conditions for the velocity field selection problem (Theorem \ref{vel_sel_necc_cond_thm}) reduce to
\begin{equation} \label{opt_v_eqn}
	v ~=~ \rho \nabla \lambda + \nabla \gamma ,
\end{equation}
and the compact form \eqref{compact_v_eqn} becomes simply $v = \Pi(\rho \nabla \lambda)$. Likewise, the Leray projection admits the familiar explicit form
\begin{equation}
	\Pi ~=~ I - \nabla (\nabla \cdot \nabla)^{-1} \nabla \cdot ~=~ I - \nabla \Delta^{-1} \nabla \cdot .
\end{equation}

Specializing the necessary conditions for the optimal state-transfer problem (Theorem \ref{state_transfer_necc_conds_thm}) to $K = I$ gives, after substituting $v = \Pi(\rho \nabla \lambda)$,
\begin{align}
	\partial_t \rho ~&=~ - \nabla \rho \cdot \Pi (\rho \nabla \lambda) , \label{necc_cond_eqn_p} \\
	\partial_t \lambda ~&=~ - \nabla \lambda \cdot \Pi (\rho \nabla \lambda) . \label{necc_cond_eqn_l}
\end{align}
As before, $\rho$ and $\lambda$ are transported together by $v = \Pi(\rho \nabla \lambda)$, and the evolution of $v$ is governed by a single equation. Specializing the general evolution equation \eqref{v_dynamics_eqn} to this case, the momentum $\mathfrak{m} = K^*K v = v$, and the skew term reduces via the identity
\begin{equation}
	2 \, \text{skew}(Dv)[v] ~=~ Dv[v] - (Dv)^\top[v] ~=~ Dv[v] - \tfrac{1}{2} \nabla \| v \|^2 ,
\end{equation}
since the pure-gradient term is annihilated by $\Pi$, to yield
\begin{equation} \label{euler_eqn}
	\partial_t v ~=~ - \Pi \big( Dv[v] \big) .
\end{equation}
That is, the velocity field evolves according to the \emph{incompressible Euler equation}.

This is a remarkable result: the geodesic equation for the reachable set $(\Reach, m)$ under the kinetic energy penalty is the continuity equation driven by a velocity field obeying the incompressible Euler equation. The interpretation of the incompressible Euler equation as a geodesic equation was first described by Arnold in the flow map setting (i.e., in the space $\sdiffomega$) in his 1966 seminal paper \citep{Arnold1966}. Here, we instead characterize this equation as a geodesic equation in the space of \emph{densities}. These two perspectives connect via the Riemannian submersion discussed in Section \ref{lagrangian_lift_sec}.

It is also interesting to point out that the system \eqref{necc_cond_eqn_p}-\eqref{necc_cond_eqn_l} provides a \emph{factoring} of the incompressible Euler equation into a coupled system of scalar-valued PDEs. In fact, this factoring is precisely the classical \emph{Clebsch representation} of an incompressible flow \citep[see, e.g.,][]{marsden1983coadjoint}, in which the velocity field $v$ is written in the form
\begin{equation}
    v ~=~ \nabla \gamma + \rho \nabla \lambda ,
\end{equation}
where, as in \eqref{necc_cond_eqn_p}-\eqref{necc_cond_eqn_l}, both potentials $\rho$ and $\lambda$ are advected by $v$, and $\gamma$ is fixed by incompressibility. Notably, the Clebsch potentials are not \emph{posited} here but \emph{derived}. In the classical setting, the potentials $\gamma, ~\rho, ~\lambda$ are introduced as an ansatz; here, they emerge naturally out of the structure of the optimal control problem, appearing as the incompressibility multiplier $\gamma$, transported density $\rho$, and costate $\lambda$.

We also point out a close connection between this geodesic equation and the classical theory of optimal mass transport: while in our setting, geodesics are governed by the \emph{incompressible} Euler equation, in the standard picture of optimal transport, geodesics are instead governed by the \emph{pressureless} Euler equation \citep[see, e.g.,][]{Benamou2000}:
\begin{equation} \label{pressureless_euler_eqn}
	\partial_t v ~=~ - D v [v] .
\end{equation}
The two differ precisely by the Leray projection $\Pi$, which enforces incompressibility.

Now, while we find these connections to be fascinating, they also reveal that the kinetic energy penalty is an unnatural choice for mixing effort, for two reasons. First, the incompressible Euler equation is notoriously poorly behaved: one expects \citep[following][]{Brenier1989} that optimal solutions to our mixing problem do not even exist in the $L^2$ case, and that the problem must be relaxed into the space of \emph{generalized flows} (i.e., Young measures). Since these are nonphysical (and extremely irregular) solutions, there is no hope of building a machine to implement -- or even approximate -- them. Second, the incompressible Euler equation is a conservative\footnote{We neglect here certain weak solutions to the incompressible Euler equation which exhibit so-called ``anomalous energy dissipation''.} physical law, meaning that the optimal mixing protocol in this case dumps all of its energy into the fluid at time $t = 0$ and does not interfere with it thereafter. This is certainly not how we expect real-world mixing machines to operate.

The ill-posedness of this problem under the kinetic energy penalty is itself informative. Because we consider optimal mixing fields independently of any parameterization, the behavior of the resulting problems tells us something about how the objective should be designed. Here, it tells us that the kinetic energy penalty is an unnatural measure of mixing effort -- something that would not have been apparent had we worked with finitely-parameterized velocity fields. We expect the enstrophy-penalized ($K = D$) case to be much better behaved, and we investigate it next in Section~\ref{example_enstrophy_sec}.

\section{Special Case: Enstrophy-Penalized Flow Fields} \label{example_enstrophy_sec}

We now specialize the general theory of Section \ref{general_k_sec} to the case where $v$ is penalized based on its enstrophy\footnote{In fluid mechanics, the term ``enstrophy'' traditionally refers to the (squared) $L^2$ norm of the vorticity, $\| \nabla \times v \|_{L^2}^2$. For incompressible velocity fields on a flat domain satisfying the no-slip condition $v = 0$ on $\partial \Omega$ (if $\Omega$ has boundary), the enstrophy $\| \nabla \times v \|_{L^2}^2$, viscous dissipation energy $2 \| S \|_{L^2}^2$ (with $S$ the symmetric rate-of-strain), and the (squared) $L^2$ norm of the Jacobian $\| D v \|_{L^2}^2$ all coincide. On a curved Riemannian manifold, these identities require more care.}, i.e., where $\| \cdot \| = \| D \cdot \|_{L^2}$ or $K = D$.

In comparison to the energy-penalized case investigated earlier, we expect this case to be much better behaved. First, optimal solutions should have some degree of regularity\footnote{Indeed, the setting that we develop here falls within the theory of regular Lagrangian flows \citep[see, e.g.,][]{ambrosio2014continuity}.}, since we penalize them based on their derivative\footnote{It is also interesting to note here that this penalty coincides with the spatial component of Brockett's ``attention functional'' \citep{Brockett1997} -- a proxy for the spatial resolution required by the sensing/actuation apparatus in order to implement such a control.}. Thus, it is more hopeful that physical machines can be built to approximate these flows. Second, the enstrophy norm is a much better proxy for mixing effort: it is proportional to the power dissipated by viscosity, which we expect to be constant. This suggests that optimal mixing in this norm occurs with constant power input, which is much more aligned with how we expect real-world mixing to work.

We begin by describing the objects which appear in this setting, and then verify our assumptions on $\V$ and $K$ from the previous sections. First, $K = D$ is the derivative operator acting on vector fields, so $Kv = Dv$ is a matrix field -- the Jacobian field of $v$. The $L^2$ (Frobenius) inner product on matrix fields is written in any of the equivalent forms $\lb M , N \rb_{L^2}$, $\int_\Omega M \! : \! N$, or $\int_\Omega \tr(M^\top N)$. The space $\V$ must then consist of vector fields with $L^2$ weak derivative, i.e., (vector-valued) $H^1$.

Recall from earlier that $K^*$ is the $L^2$-adjoint of $K$, which is given in this setting by the (negative) \emph{matrix divergence} $D^*$. In index notation,
\begin{equation}
	(D^* M)_i ~=~ - \text{div}(M)_i ~=~ - \sum_j \partial_j M_{ij} .
\end{equation}
The operator $K^*K = D^*D$ is then the \emph{vector Laplacian}\footnote{Note that the operator $\Delta$ as usually defined is negative definite, while $-\Delta = D^*D$ is positive definite. Our convention is thus to work with $-\Delta$, which is the more natural operator here.} $-\Delta$, which acts on $v$ as a component-wise scalar Laplacian. Its inverse $(-\Delta)^{-1}$ is exactly the solution operator of the vector Poisson equation:
\begin{equation}
	u ~=~ (-\Delta)^{-1} f \qquad \Leftrightarrow \qquad -\Delta u = f
\end{equation}
(and numerically, such inverses are typically computed this way).

Unlike the energy-penalized case -- where $v \in L^2$ does not have well-defined boundary values -- $v \in H^1$ does, and it is natural to impose a no-slip boundary condition $v = 0$ on $\partial \Omega$ when $\Omega$ has nonempty boundary. When $\Omega$ is boundaryless, we instead quotient out by the kernel of $D$ (i.e., the parallel vector fields). On the flat $n$-torus, for example, the parallel vector fields are exactly the constant vector fields, and this quotient is equivalent to fixing $v$ to be zero-mean $\int_\Omega v = 0$ in standard coordinates. In either case, we denote the resulting space $\dot{H}^1$, for homogeneous (vector valued) $H^1$, and write $\dot{H}^{-1}$ for its dual.

We again retain our prior assumptions on $\Omega$ and $\mu \in \{ \rho, \1 \}$: $\Omega$ is a compact, connected, $n$-dimensional Riemannian manifold with Lipschitz (or empty) boundary, and $\rho$ is measurable and uniformly bounded above and below by positive constants: $0 < c \leq \rho \leq C < \infty$. We also make the additional assumption here that $\rho$ (and thus $\mu \in \{ \rho, \1\}$) is Lipschitz continuous.

The space $\H$ of scalar potentials is taken to be $L^2_0$, the space of $L^2$ scalar fields satisfying the zero-mean condition $\int_\Omega \xi = 0$. With all of these objects in place, we may now verify our assumptions.

\begin{prop} \label{enstrophy_assumptions_prop}
    Assumptions \ref{helmholtz_assumptions} hold under the choices: $\V = \dot{H}^1$, $K=D$, and $\H = L^2_0$.
\end{prop}

\begin{proof}
    Removing the kernel of $K$ (either by boundary condition or quotient) makes $K$ injective. $K$ is also densely defined since $\dot{H}^1$ is dense\footnote{In the boundaryless case, we need to work in the ambient space $L^2 / \Nu(D)$ rather than all of $L^2$, but this does not affect the resulting treatment -- the load-bearing conclusions still follow and lead to identical formulas.} in $L^2$. Removal of the kernel also means that the vector-valued Poincar\'e inequality holds, implying that the norms $\| \cdot \|_{H^1} := (\| \cdot \|_{L^2}^2 + \| D \cdot \|_{L^2}^2)^{1/2}$ and $\| \cdot \|_{\dot{H}^1} := \| D \cdot \|_{L^2}$ are equivalent, and making $\V = \dot{H}^1$ a Hilbert space. Since $\V$ is complete and $K: \dot{H}^1 \to L^2$ is an isometry (by definition of $\| \cdot \|_{\dot{H}^1}$), $K$ has closed range.

    The Lipschitz assumption on $\mu$ ensures that multiplication by $\mu$ is a continuous operation from $\dot{H}^1 \to \dot{H}^1$ and from $\dot{H}^{-1} \to \dot{H}^{-1}$ (as well as from $L^2 \to L^2$). This allows us to verify the hypotheses of Lax-Milgram as follows:
    \begin{align}
        | a(\psi,\xi) | ~&=~ |\lb \mu \nabla \psi , (-\Delta)^{-1} \mu \nabla \xi \rb_{L^2}| ~=~ | \lb \mu \nabla \psi , \mu \nabla \xi \rb_{\dot{H}^{-1}} | \nonumber \\
        &\qquad \qquad ~\leq~ \| \mu \nabla \psi \|_{\dot{H}^{-1}} \| \mu \nabla \xi \|_{\dot{H}^{-1}} ~\lesssim~ \| \psi \|_{L^2} \| \xi \|_{L^2} ; \\
        a(\xi, \xi) ~&=~ \lb \mu \nabla \xi , (-\Delta)^{-1} \mu \nabla \xi \rb_{L^2} ~=~ \| \mu \nabla \xi \|_{\dot{H}^{-1}}^2 ~\gtrsim~ \| \xi \|_{L^2}^2 ; \\
        | f(\psi) | ~&=~ | \lb \mu \nabla \psi , v \rb_{L^2} | ~=~ | \lb \psi , \nabla \cdot (\mu v) \rb_{L^2} | \nonumber \\
        & \qquad \qquad ~\leq~ \| \nabla \cdot (\mu v) \|_{L^2} \| \psi \|_{L^2} ~\lesssim~ \| v \|_{\dot{H}^1} \| \psi \|_{L^2} ,
    \end{align}
    where we have used the identities $\lb f , (-\Delta)^{-1} g \rb_{L^2} = \lb f , g \rb_{\dot{H}^{-1}}$ and $\| \nabla f \|_{\dot{H}^{-1}} = \| f \|_{L^2}$. In particular, hypothesis (iii) identifies $\nabla \cdot (\mu \ph)$ as continuous from $\dot{H}^1$ into $L^2_0$.
\end{proof}

With these assumptions verified, we now specialize the necessary conditions from the previous section. As in the energy-penalized case, we first record the form taken by the Helmholtz projection $\Pi = \Pi_{\1,D}$. From Section \ref{general_k_sec}, this projection admits the explicit form
\begin{equation}
	\Pi_{\1,D} ~=~ I - (-\Delta)^{-1} \nabla \big( \nabla \cdot (-\Delta)^{-1} \nabla \big)^{-1} \nabla \cdot .
\end{equation}
Under certain circumstances, this expression simplifies. Recall the vector calculus identity $\Delta v = \nabla (\nabla \cdot v) - \nabla \times (\nabla \times v)$. If $\nabla \times (\nabla \times \nabla f) = 0$ for all scalar $f$ (i.e., if the relevant de Rham complex is exact), then the Laplacian and gradient commute:
\begin{equation}
	\Delta \nabla f ~=~ \nabla (\nabla \cdot \nabla f) - \nabla \times (\nabla \times \nabla f) ~=~ \nabla \Delta f .
\end{equation}
This holds, for example, when $\Omega$ is the flat $n$-torus with the zero-mean condition on $\V$. In this case, since $\nabla$ commutes with $\Delta$ (and hence with $(-\Delta)^{-1}$), the projection simplifies to the standard Leray projection:
\begin{align}
	\Pi_{\1,D} ~&=~ I - (-\Delta)^{-1} \nabla \big( \nabla \cdot (-\Delta)^{-1} \nabla \big)^{-1} \nabla \cdot \\
	~&=~ I -  \nabla (-\Delta)^{-1} \big( \nabla \cdot \nabla (-\Delta)^{-1} \big)^{-1} \nabla \cdot \\
	~&=~ I -  \nabla \Delta^{-1} \nabla \cdot \\
	~&=~ \Pi_{\1,I} .
\end{align}

With the projection in hand, the necessary conditions specialize directly from Theorems \ref{vel_sel_necc_cond_thm} and \ref{state_transfer_necc_conds_thm} with $K^*K = -\Delta$. For the velocity field selection problem, we obtain
\begin{equation}
	v ~=~ (-\Delta)^{-1}(\rho \nabla \lambda + \nabla \gamma) ~=~ \Pi_{\1,D} \big( (-\Delta)^{-1} (\rho \nabla \lambda) \big) ,
\end{equation}
and for the optimal state-transfer problem, the coupled system
\begin{align}
	\partial_t \rho ~&=~ - \nabla \rho \cdot \Pi_{\1,D} \big( (-\Delta)^{-1} (\rho \nabla \lambda) \big) , \label{rho_enstrophy_geodesic_eqn} \\
	\partial_t \lambda ~&=~ - \nabla \lambda \cdot \Pi_{\1,D} \big( (-\Delta)^{-1} (\rho \nabla \lambda) \big) , \label{lambda_enstrophy_geodesic_eqn}
\end{align}
with the corresponding evolution equation for $v$,
\begin{equation} \label{enstrophy_v_dynamics_eqn}
	\partial_t v ~=~ - \Pi_{\1,D} \big( (-\Delta)^{-1} \big( 2 \, \operatorname{skew} (D(-\Delta v))[v] \big) \big) .
\end{equation}

Unlike the energy-penalized case -- where the evolution equation collapsed to the incompressible Euler equation -- this equation does not simplify, even in the case where $\nabla$ and $\Delta$ commute and $\Pi_{\1,D}$ reduces to the standard Leray projection. Following our earlier remark in Section \ref{general_k_sec}, Equation \eqref{enstrophy_v_dynamics_eqn} is the Euler-Poincar\'e (EPDiff) geodesic equation for the right-invariant homogeneous $\dot{H}^1$ metric on $\sdiffomega$, with inertia operator $K^*K = -\Delta$. To the best of our knowledge, its solutions have not been studied, though it represents a limiting case of the \emph{Lagrangian-averaged Euler} (LAE-$\alpha$) family of equations, or alternatively, a degenerate case in the classification of right-invariant Sobolev $H^1$ metrics due to \citet[Corollary~A.5]{khesin2013geometry}. We provide numerical solutions to this equation in the following subsection.

\subsection{Numerical Simulations} \label{enstrophy_numerical_sec} \label{geodesic_simulations_sec}

We solve the geodesic equations \eqref{rho_enstrophy_geodesic_eqn}-\eqref{lambda_enstrophy_geodesic_eqn} on the domain $\Omega = [0,2\pi) \times [0,2\pi)$ with periodic boundary conditions, i.e., on the flat 2-torus $\mathbb{T}^2$. The domain is discretized on an $N \times N$ grid with spacing $\delta x = \delta y = 2\pi/N$, and we write $v = (U, V)$ for the components of the velocity field.

The advection equations $\partial_t u = -\nabla u \cdot v$ for $u \in \{\rho, \lambda\}$ are solved with a first-order upwind finite difference scheme, written in conservative flux form so that discrete mass is preserved exactly. %
Because $v$ varies in both space and time, the upwind direction is chosen adaptively at every grid point and every timestep. The timestep is also chosen adaptively to satisfy the Courant--Friedrichs--Lewy (CFL) condition. Since $\delta x = \delta y$, this takes the form
\begin{equation} \label{cfl_condition_eqn}
    \delta t ~=~ \min\left( \delta t_{\max}, ~ \frac{\text{C} \cdot \delta x}{\max_{i,j} \big( |U_{i,j}| + |V_{i,j}| \big)} \right) .
\end{equation}

Each advection step alternates with an update of the velocity field $v = \Pi [\Delta^{-1}(\rho \nabla \lambda)]$. We first compute $w := \rho \nabla \lambda$, approximating $\nabla \lambda$ with second-order centered differences. Since the kernel of $\Delta$ on $\mathbb{T}^2$ consists of the constant functions, we then remove the mean of $w$:
\begin{equation} \label{zero_mean_center_eqn}
    \tilde{w} ~:=~ w - \frac{1}{|\Omega|} \int_\Omega w \, dx \, dy .
\end{equation}
The inverse vector Laplacian and the Leray projection $\Pi = I - \nabla \Delta^{-1} \nabla$ are diagonal in Fourier space, so both are applied in a single step:
\begin{equation} \label{leray_fourier_eqn}
    \hat{v}(\xi) ~=~ -\frac{1}{|\xi|^2} \left( I - \frac{\xi \xi^{\mathsf{T}}}{|\xi|^2} \right) \hat{\tilde{w}}(\xi), \qquad \xi \in \mathbb{Z}^2 \setminus \{0\} ,
\end{equation}
with $\hat{v}(0) = 0$. Here $\hat{v}$ and $\hat{\tilde{w}}$ are the discrete Fourier transforms of $v$ and $\tilde{w}$ and $\xi = (\xi_1, \xi_2)$ is the integer wavenumber. %
An inverse Fourier transform then returns the divergence-free velocity field on the grid.

Figure~\ref{fig:geodesic_evolution} shows the evolution from the initial conditions
\begin{equation}
    \rho(x,y) ~=~ e^{1.5 \sin(x+y)}, \qquad \lambda(x,y) ~=~ e^{1.5 \cos(x - y + 1.3)} ,
\end{equation}
both of which are smooth and $2\pi$-periodic in each variable. We used $N = 256$ (so $\delta x = 2\pi/256 \approx 0.0245$), $C = 0.9$, and a final time $T = 5$. Over this run the $L^2$ norm of $\rho$ changes by $-5.91\%$; in the continuous problem the latter is conserved so this drift reflects the numerical diffusion of the first-order scheme. Total mass is conserved to machine precision.

To check that this drift is a genuine discretization effect we repeated the same run at $N = 64, 128, 256, 512$, holding $C = 0.9$ and $T = 5$ fixed, and tracked how the $L^2$ change reported above depends on $N$ (Table~\ref{tab:geodesic_resolution}).

\begin{table}[ht]
\centering
\caption{Resolution dependence of the change in $\|\rho\|_{L^2}$ over the geodesic run of Figure~\ref{fig:geodesic_evolution} (CFL number $C = 0.9$, final time $T = 5$). Here $N$ is the number of grid points in each direction and $\delta x = 2\pi/N$ is the grid spacing.}
\label{tab:geodesic_resolution}
\begin{tabular*}{0.61\textwidth}{@{\extracolsep{\fill}} c c c}
\toprule
$N$ & $\delta x$ & Change in $\|\rho\|_{L^2}$ \\
\midrule
64  & $9.82 \times 10^{-2}$ & $-10.03\%$ \\
128 & $4.91 \times 10^{-2}$ & $-7.82\%$ \\
256 & $2.45 \times 10^{-2}$ & $-5.91\%$ \\
512 & $1.23 \times 10^{-2}$ & $-4.31\%$ \\
\bottomrule
\end{tabular*}
\end{table}

The drift shrinks steadily as $N$ increases, which is what we expect from the numerical diffusion of a first-order scheme. It shrinks more slowly than a first-order scheme normally would though: each time $N$ doubles, the drift drops by only about a quarter instead of being cut in half. We attribute this to the filamentation visible in Figure~\ref{fig:geodesic_evolution}: as $\rho$ and $\lambda$ are sheared into progressively thinner bands, the local gradients being advected grow with $N$, which degrades the effective accuracy of the scheme relative to a fixed smooth test problem. The code used to generate these figures is publicly available at \url{https://github.com/johnigraszekcode/Fluid-Mixing-Code}.

\begin{figure}[ht]
    \centering
    \includegraphics[width=\textwidth]{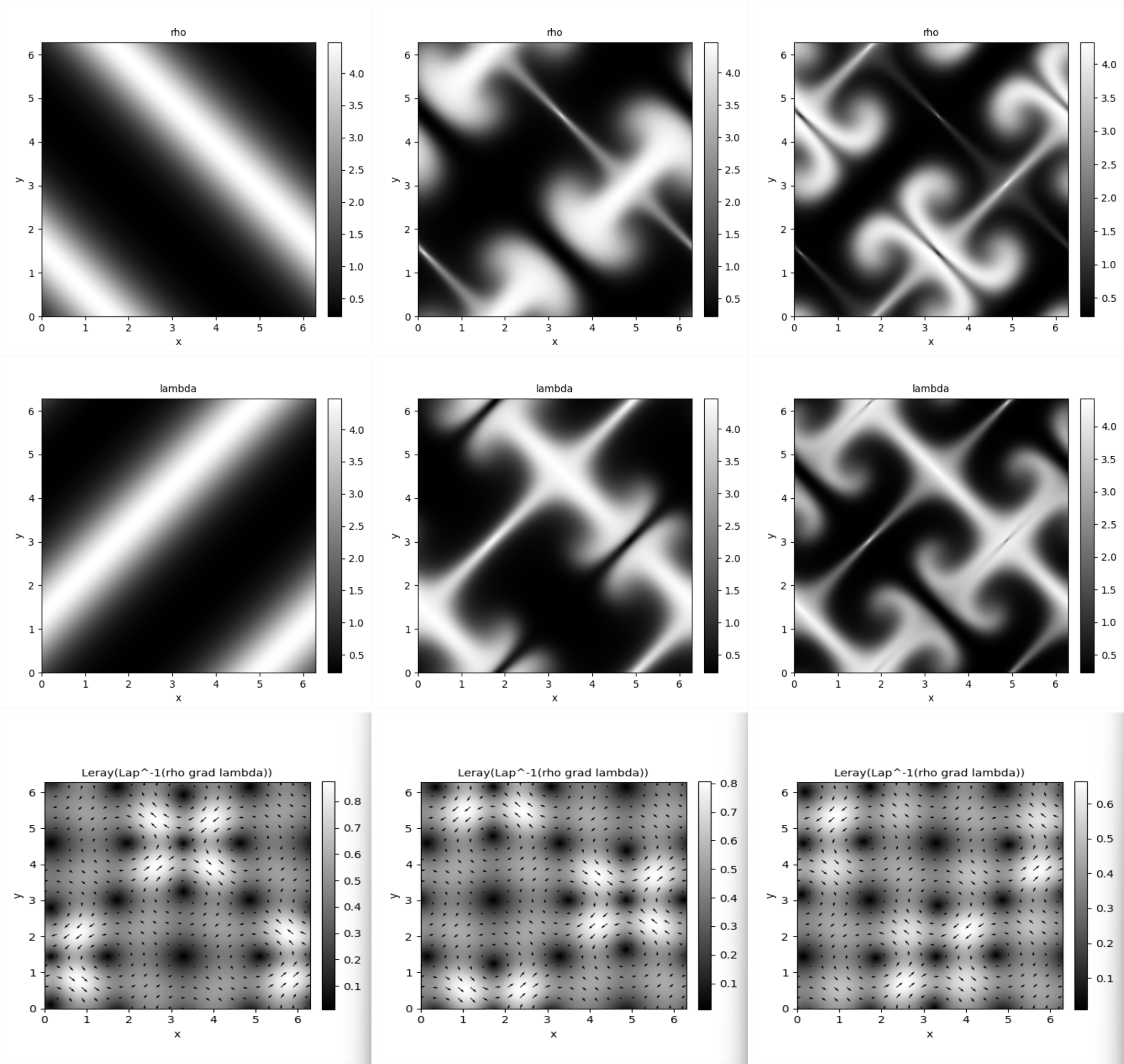}
    \caption{Geodesic evolution on $\mathbb{T}^2$, shown at three equally spaced times from $t = 0$ to $t = 5$ (left to right). The rows show the density $\rho$ (top), the costate $\lambda$ (middle), and the velocity $v$ (bottom). Both $\rho$ and $\lambda$ start out as smooth diagonal bands running in roughly perpendicular directions. The velocity generated by $\rho \nabla \lambda$ shears the bands of $\rho$ against those of $\lambda$, stretching them into progressively thinner filaments.
    }
    \label{fig:geodesic_evolution}
\end{figure}

\section{Gradient Flows} \label{gradient_flows_sec}

As another application of the Riemannian structure described in Section \ref{riemannian_structure_sec}, we now define \emph{gradient flows} on the reachable set $(\Reach,m)$. Gradient flows provide yet another approach to control for mixing. Whereas the full optimal control solutions \eqref{necc_cond_dyn_eqn}-\eqref{necc_cond_v_eqn} are globally optimal between their endpoints but require expensive iterative numerical solution due to the split boundary conditions, gradient flows can be viewed as locally optimal or ``greedy'' solutions -- not nearly as efficient, but much easier to compute, since they may simply be integrated forward in time.

Given a Riemannian manifold $M$ with metric $g$ and an energy function $E: M \to \R$, the gradient flow\footnote{\label{metric_gradient_flow_footnote}It is worth noting that there is also a well-developed theory of gradient flows in \emph{metric spaces} \citep{Ambrosio2005} which does not rely on the Riemannian structure. What the Riemannian structure buys us is explicit formulas for these gradient flows.} of $E$ on $M$ is written
\begin{equation} \label{riemannian_gradient_flow_eqn}
    \frac{d}{dt} x ~=~ - \grad_g E(x) ,
\end{equation}
where $\grad_g$ denotes the \emph{Riemannian gradient}, defined as the Riesz representation of the first variation. That is, the first variation $\delta E$ is defined as the linear functional
\begin{equation}
    \delta E(x)[u] ~:=~ \lim_{s \to 0} \frac{1}{s} \big( E(\eta_s) - E(\eta_0) \big) ,
\end{equation}
where $\eta: [-\varepsilon,\varepsilon] \to M$ is any curve with $\eta_0 = x \in M$ and $\eta_0' = u \in T_x M$. The Riemannian gradient $\grad_g E$ is then the unique vector field on $M$ satisfying
\begin{equation} \label{grad_def_eqn}
    \lb \grad_g E(x) , u \rb_x ~=~ \delta E(x)[u]
\end{equation}
for all $x \in M$ and $u \in T_xM$.

We now compute an explicit expression for the Riemannian gradient in our setting, taking $M = (\Reach,m)$. Let $\rho: [-\varepsilon,\varepsilon] \to \Reach$ be a curve in the reachable set, generated by a velocity field $v$. Without loss of generality, we may take $v$ to be optimal at each time (see Sections \ref{riemannian_structure_sec} and \ref{helmholtz_sec}), so that the tangent vector $\tau = \partial_t \rho = -\nabla \cdot (\rho v)$ has $v$ as its minimal representative. Let $f: \Reach \to \R$ be a scalar function, and evaluate the right-hand side of \eqref{grad_def_eqn} by differentiating $f$ along $\rho$, applying the dynamics, and integrating by parts:
\begin{align}
    \delta f(\rho)[\partial_t \rho] ~&=~ \lb \grad_{L^2} f(\rho) , - \nabla \cdot (\rho v) \rb_{L^2} \nonumber \\
    ~&=~ \lb \rho \nabla \grad_{L^2} f(\rho) , v \rb_{L^2} \nonumber \\
    ~&=~ \lb (K^*K)^{-1} \rho \nabla \grad_{L^2} f(\rho) , v \rb_\V .
\end{align}
Here, $\grad_{L^2} f(\rho)$ is the $L^2$-Riemannian gradient of $f$, sometimes also called the \emph{functional derivative}. Note that this object is a scalar field on $\Omega$ (for fixed $\rho$), but a vector field on $\Reach$ (as a function of $\rho$), consistent with our identification of tangent vectors in Section \ref{riemannian_structure_sec}. Recall that since $v$ is optimal, it must satisfy $v = \Pi_{\1,K} v$, and using the fact that $\Pi_{\1,K}$ is a $\V$-orthogonal (hence $\V$-self-adjoint) projection,
\begin{align}
    \delta f(\rho)[\partial_t \rho] ~&=~ \lb (K^*K)^{-1} \rho \nabla \grad_{L^2} f(\rho) , \Pi_{\1,K} v \rb_\V \nonumber \\
    ~&=~ \lb \Pi_{\1,K} (K^*K)^{-1} \rho \nabla \grad_{L^2} f(\rho) , v \rb_\V .
\end{align}
Both arguments in the above inner product are now optimal representatives for their respective equivalence classes -- $v$ by assumption, and the left argument by virtue of its form $\Pi_{\1,K}(K^*K)^{-1}\rho\nabla\xi$. Thus, this $\V$-inner product coincides with the Riemannian metric \eqref{inner_prod_eqn}:
\begin{equation}
    \delta f(\rho)[\partial_t \rho] ~=~ \big\lb \Pi_{\1,K} (K^*K)^{-1} \rho \nabla \grad_{L^2} f(\rho) , \, v \big\rb_{\rho} .
\end{equation}
Comparing with the defining property \eqref{grad_def_eqn}, we identify the Riemannian gradient, represented by its minimal velocity field, as
\begin{equation} \label{grad_m_formula_eqn}
    \grad_m f(\rho) ~=~ \Pi_{\1,K} \left[ (K^*K)^{-1} \rho \nabla \grad_{L^2} f(\rho) \right] .
\end{equation}
As a tangent vector at $\rho$, this gradient $\grad_m f(\rho)$ may be represented either by the minimal velocity field as above, or by the corresponding density perturbation $\tau = -\nabla \cdot (\rho \, \grad_m f(\rho) )$.

With the above derivation, we have just proven the following.

\begin{thm}
    The gradient flow of the energy $E$ on $\Reach$ with respect to the metric $m$ is given by the PDE
    \begin{equation} \label{gradient_flow_pde_eqn}
        \partial_t \rho ~=~ \nabla \cdot \left( \rho \, \Pi_{\1,K} \left[ (K^*K)^{-1} \rho \nabla \grad_{L^2} E(\rho) \right] \right) .
    \end{equation}
\end{thm}

\begin{proof}
    Follows directly by applying the density perturbation form of the gradient \eqref{grad_m_formula_eqn} to the definition of the gradient flow \eqref{riemannian_gradient_flow_eqn}.
\end{proof}

It is also worth noting that the gradient flow dynamics \eqref{gradient_flow_pde_eqn} satisfy the following \emph{energy dissipation equality} (EDE)
\begin{align} \label{ede_eqn}
    \frac{d}{dt} E(\rho) ~&=~ \delta E(\rho)[\partial_t \rho] ~=~ \lb \grad_m E(\rho) , \partial_t \rho \rb_{\rho} \nonumber \\
    ~&=~ \lb \grad_m E(\rho) , -\grad_m E(\rho) \rb_{\rho} ~=~ -\big\| \grad_m E(\rho) \big\|_{\rho}^2 ,
\end{align}
sometimes taken as an alternative \emph{definition} of gradient flow in the metric space setting (see Footnote \ref{metric_gradient_flow_footnote}).

Gradient flows for mixing may then be constructed simply by taking the energy functional $E$ to be a measure of mixedness, say $E(\rho) := \tfrac{1}{2} \, d^2(\rho,\rho_*)$, (half) the squared distance to the target density $\rho_*$ (typically taken to be the ``fully mixed'' uniform density with the same mass as $\rho$). By construction of the metric $m$, the resulting trajectory then decreases $d(\rho,\rho_*)$ as much as possible for any fixed mixing effort at each instant in time.

It is also interesting to remark on the shared structure between the gradient flow dynamics \eqref{gradient_flow_pde_eqn} and the geodesic equations \eqref{gen_k_dynamics_1}-\eqref{gen_k_dynamics_2} from before. Both describe the evolution of the density $\rho$, under the continuity equation $\partial_t \rho = - \nabla \cdot (\rho v)$, driven by a vector field of the form $v = \Pi_{\1,K} \left( (K^*K)^{-1} \rho \nabla \xi \right)$. As we have seen, a vector field of this form is always the optimal representative for its equivalence class. The difference between the two equations is where the scalar potential $\xi$ comes from. In the geodesic equations, $\xi$ is the costate $\lambda$, which has an interpretation as the derivative of the cost-to-go (defined via the Hamilton-Jacobi-Bellman equation). In the gradient flow equation, $\xi$ is the $L^2$ gradient (i.e. the functional derivative) of the energy $\grad_{L^2} E(\rho)$. These two objects coincide precisely in the case where $E(\rho) = m(\rho,\rho_f)$, i.e., where the energy equals the cost-to-go which equals the distance to the final state.

In the next section, we provide numerical simulations of these gradient flows.

\subsection{Numerical Simulations} \label{gradient_flow_simulations_sec}
We solve the gradient flow equation \eqref{gradient_flow_pde_eqn} in the enstrophy-penalized setting on the flat torus $\mathbb{T}^2$, reusing the discretization of Section~\ref{geodesic_simulations_sec} almost unchanged. The only structural difference is that there is no costate to evolve, so instead of advecting $\lambda$ alongside $\rho$ we compute the first variation $\delta E(\rho) = \grad_{L^2} E(\rho)$ from $\rho$ at every timestep and use it in place of $\lambda$ in the velocity formula
\begin{equation} \label{gradient_flow_velocity_eqn}
    v ~=~ \Pi[ \Delta^{-1} \big( \rho \nabla \delta E(\rho) \big) ].
\end{equation}

We take $E(\rho)$ to be half the squared $\dot{H}^{-1}$ (homogeneous negative-order Sobolev) distance between $\rho$ and $\rho^*$, where $\rho^*$ is the uniform density with the same mass as $\rho$. Equivalently, $E$ is half the squared $\dot{H}^{-1}$ norm of the zero-mean part $\tilde{\rho}$ of $\rho$:
\begin{equation}
    E(\rho) ~:=~ \frac{1}{2} \, \| \rho - \rho^* \|_{\dot{H}^{-1}}^2 ~=~ \frac{1}{2} \, \| \tilde{\rho} \|_{\dot{H}^{-1}}^2 .
\end{equation}
On $\mathbb{T}^2$, this energy has the Fourier representation
\begin{equation}
    E(\rho) ~=~ \frac{1}{2} \, \big\| \nabla \Delta^{-1} \tilde{\rho} \big\|_{L^2}^2 ~=~ \frac{1}{2} \sum_{\xi \neq 0} \frac{1}{|\xi|^2} \, |\hat{\tilde{\rho}}(\xi)|^2 ,
\end{equation}
with first variation
\begin{equation}
    \delta E(\rho) ~=~ -\Delta^{-1} \tilde{\rho} , \qquad \widehat{\delta E(\rho)}(\xi) ~=~ \frac{1}{|\xi|^2} \, \hat{\tilde{\rho}}(\xi) .
\end{equation}
This choice is convenient numerically, since it costs only one additional scalar Poisson solve per timestep on top of the machinery already in place. Each velocity update computes $\delta E(\rho)$ in Fourier space, transforms back to physical space to form $\rho \nabla \delta E(\rho)$ using centered differences, and then applies \eqref{zero_mean_center_eqn}-\eqref{leray_fourier_eqn} exactly as in Section~\ref{geodesic_simulations_sec}. 
We then use exactly the same conservative upwind finite difference scheme from Section~\ref{geodesic_simulations_sec} to advance $\rho$.

Figure~\ref{fig:gradient_flow_evolution} shows the evolution from the initial condition
\begin{equation}
    \rho(x,y) ~=~ 2 + 1.5 \sin(x) \cos(y) + \sin(2x) \sin(2y) + 0.8 \cos(3x + y) .
\end{equation}
We used $N = 256$ (so $\delta x = 2\pi/256 \approx 0.0245$), $C = 0.45$, and a final time $T = 75$. The energy decreases at every timestep, falling by $86.2\%$ over the run (from $E = 6.80$ to $E = 0.94$), and total mass is conserved to machine precision. Meanwhile, the $L^2$ norm of $\rho$ contracts by $2.09\%$ over the run
($\|\rho_0\|_{L^2} = 14.23 \to \|\rho_f\|_{L^2} = 13.94$).

As in Section~\ref{geodesic_simulations_sec}, we verified predicted behavior across resolutions by repeating the run at $N = 64, 128, 256, 512$ with $C = 0.45$ and $T = 75$ held fixed, tracking the total energy drop (Table~\ref{tab:gradient_flow_resolution}).

\begin{table}[ht]
\centering
\caption{Resolution dependence of the relative decrease in the energy $E$ over the gradient-flow run of Figure~\ref{fig:gradient_flow_evolution} (CFL number $C = 0.45$, final time $T = 75$). Here $N$ is the number of grid points in each direction and $\delta x = 2\pi/N$ is the grid spacing.}
\label{tab:gradient_flow_resolution}
\begin{tabular*}{0.61\textwidth}{@{\extracolsep{\fill}} c c c}
\toprule
$N$ & $\delta x$ & Energy drop \\
\midrule
64  & $9.82 \times 10^{-2}$ & $87.95\%$ \\
128 & $4.91 \times 10^{-2}$ & $86.86\%$ \\
256 & $2.45 \times 10^{-2}$ & $86.16\%$ \\
512 & $1.23 \times 10^{-2}$ & $85.78\%$ \\
\bottomrule
\end{tabular*}
\end{table}

The energy drop settles down as $N$ increases, with the values at $N = 256$ and $N = 512$ already within a few tenths of a percent of one another. This stability is the main point of the check, and gives confidence that the numerical method is accurately resolving the relevant dynamics over this time horizon. 
The code used to generate these figures is publicly available at \url{https://github.com/johnigraszekcode/Fluid-Mixing-Code}.

\begin{figure}[ht]
    \centering
    \includegraphics[width=\linewidth]{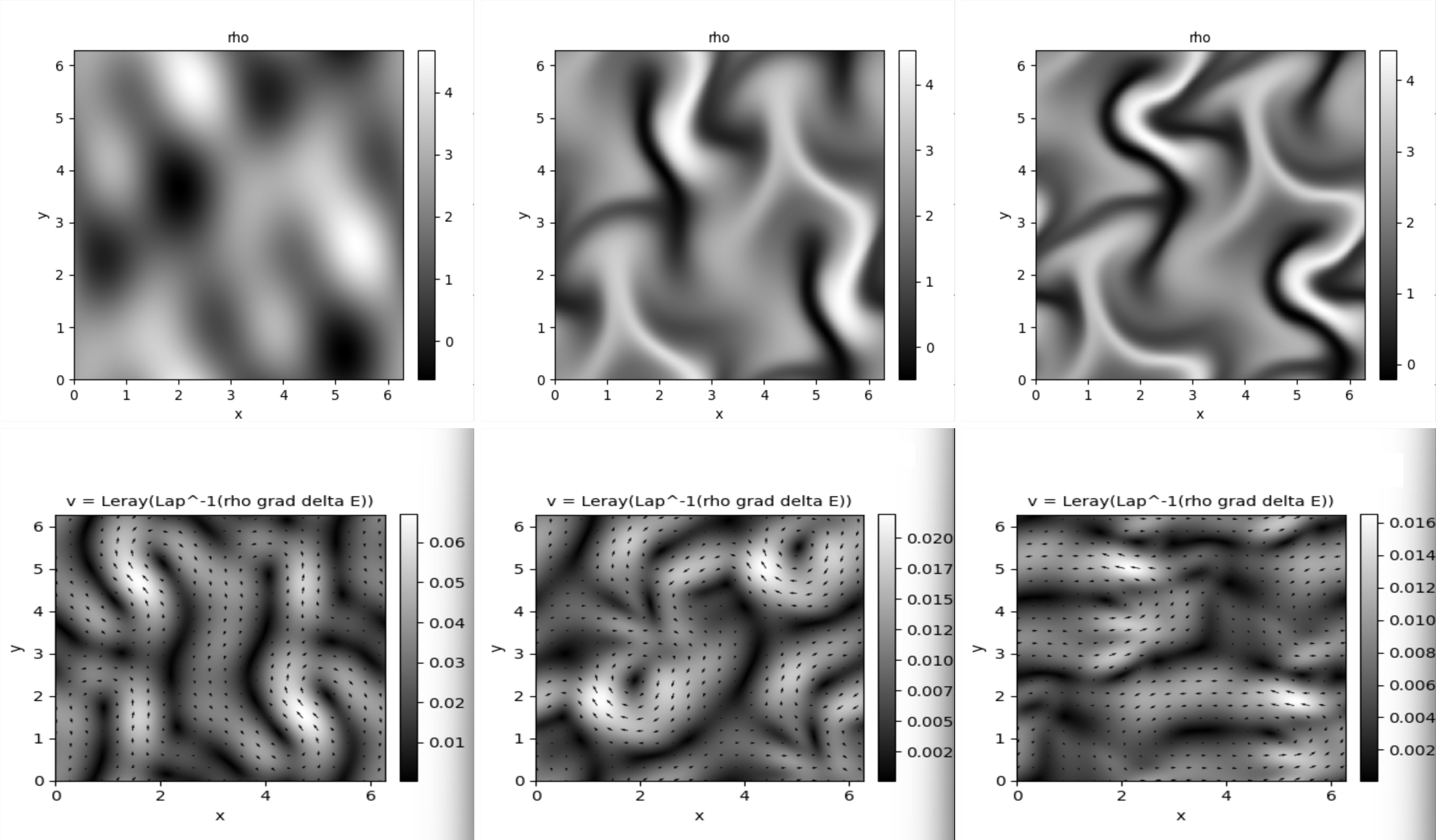}
    \caption{Gradient flow of the $\dot{H}^{-1}$ energy on $\mathbb{T}^2$ at $t = 0$, $37.5$, and $75$ (left to right). The rows show the density $\rho$ (top) and the divergence-free velocity $v$ (bottom); the velocity is shown as its magnitude, with arrows indicating direction. The initial density is a superposition of several Fourier modes at different scales. The flow quickly smooths out the small-scale structure and rearranges the density into a few sharp-fronted channels, carrying mass from high-density regions toward low-density ones. The energy decreases steadily throughout, dropping by about $86\%$, while the total mass stays fixed.}
    \label{fig:gradient_flow_evolution}
\end{figure}

\section{The Lagrangian Lift} \label{lagrangian_lift_sec}

Practically all of the problems investigated so far -- which have been posed in terms of densities -- have counterparts in terms of flow maps. In fluid mechanics, these two settings are referred to as Eulerian and Lagrangian. We describe (Section \ref{flow_maps_sec}) this correspondence, and show (Section \ref{riemannian_submersion_sec}) how the geometry of the two settings are related via Riemannian submersion. We then explain (Section \ref{lie_theory_sec}) how this Eulerian-Lagrangian correspondence and the construction of transport metrics are unified in terms of Lie group and Lie algebra actions on homogeneous spaces. This section is modeled on Arnold's theory of diffeomorphism groups \citep{Arnold1966} and Otto's theory of Riemannian calculus for optimal transport \citep{Otto2001}.

\subsection{From Densities to Flow Maps} \label{flow_maps_sec}

The problems investigated so far have been posed in the \emph{Eulerian} framework, i.e., in terms of densities $\rho$ evolving over a spatial domain $\Omega$. However, it is also possible to study these problems in a \emph{Lagrangian} framework, i.e., in terms of the evolutions of the positions of individual particles.

In the Lagrangian setting, we index (or ``label'') particles by their initial position $x \in \Omega$. The positions of these particles then evolve to follow the vector field $v$. In other words, trajectories of particles evolve along \emph{integral curves} of $v$. Letting $\phi_t(x)$ denote the position of the particle with index $x$ at time $t$, $\phi_t(x)$ evolves according to the ordinary differential equation
\begin{equation}
    \frac{d}{dt} \phi_t(x) ~=~ v_t(\phi_t(x)) ~=:~ v_t \of \phi_t(x)
\end{equation}
with initial condition $\phi_0(x) = x$. If we wish to describe the evolution of all particles at once, then the entire function $\phi$ evolves according to the PDE
\begin{equation} \label{flow_eqn}
    \partial_t \phi_t ~=~ v_t(\phi_t(\cdot)) ~=:~ v_t \of \phi_t
\end{equation}
with initial condition $\phi_0 = e$, the identity. The equation \eqref{flow_eqn} is called the \emph{flow} and the map $\phi_t: \Omega \to \Omega$ is called the \emph{flow map}. By the existence and uniqueness of ODEs (say, if $v$ is Lipschitz), the flow is well-defined and the flow map is continuous. Furthermore, it is a foundational result in differential geometry that if $v_t$ is of class $C^k$ (in space), $k \geq 1$, then $\phi_t$ is a $C^k$ diffeomorphism. The set of diffeomorphisms from $\Omega \to \Omega$ (possessing multiplication (composition), identity, and inverse) form a group -- the \emph{diffeomorphism group} of $\Omega$, denoted $\diffomega$, or simply $\D$ for short. It is thus very natural to view flow maps $\phi$ as elements of $\D$, and we adopt this viewpoint throughout. This algebraic structure -- composition and inversion of maps -- has no counterpart for densities, and it is what makes the tools of Lie theory available to us in this setting. We develop this viewpoint in Section \ref{lie_theory_sec}.

One expects that the flow map $\phi$ is a somewhat ``larger'' object than $\rho$, since it is carrying extra information about labels/positions, and this is indeed the case. While $\rho$ maps from $\Omega \to \R$, $\phi$ maps from $\Omega \to \Omega$, so if $\dim(\Omega) > 1$ (as it must be for any nontrivial mixing to occur), then $\phi$ lives in a genuinely higher-dimensional space than $\rho$. On the other hand, one expects that if we know the positions of all of the particles, then we should be able to reconstruct the density. And indeed, it holds that
\begin{equation} \label{pushforward_eqn}
	\rho_t ~=~ (\phi_t)_\# \rho_0 , 
\end{equation}
where $(\cdot)_\#$ denotes the \emph{measure pushforward}, defined via the formula
\begin{equation}
    \int_A (\phi_t)_\# \rho_0 ~=~ \int_{\phi_t^{-1}(A)} \rho_0
\end{equation}
for all $A \subset \Omega$ measurable. Intuitively, $(\phi_t)_\# \rho_0$ is the density formed by taking all of the particles in $\rho_0$ and ``pushing them forward'' according to the map $\phi_t$. This relationship -- $\phi$ encoding at least all the information that $\rho$ has, but living in a higher-dimensional space -- may be summarized by saying that $\phi$ is a \emph{lift} of $\rho$.

We can make the sense in which $\phi$ is a lift of $\rho$ more precise. Two flow maps $\phi_1$ and $\phi_2$ push $\rho_0$ forward to the \emph{same} density -- $(\phi_1)_\# \rho_0 = (\phi_2)_\# \rho_0$ -- if and only if they are related by a measure-preserving transformation: there exists a map $\omega$ with $\omega_\# \rho_0 = \rho_0$ such that $\phi_1 \of \omega = \phi_2$. We refer to such maps $\omega$ as measure-preserving \emph{relabelings}, since they reindex the particles according to $x \leftarrow \omega(x)$ without altering the initial density $\rho_0$. Thus each density $\rho$ corresponds not to a single flow map but to an entire equivalence class of flow maps, equivalent modulo measure-preserving relabeling. In other words, the space of densities is recovered from the space of flow maps by quotienting out by relabeling. We will return to this correspondence in the next two subsections.

The dynamics in the Lagrangian setting are then given by the flow \eqref{flow_eqn} together with the incompressibility constraint
\begin{equation}
    \nabla \cdot v ~=~ 0 .
\end{equation}
Earlier, we identified several conserved quantities for $\rho$ arising from this constraint. On the other hand, $\phi$ admits just (one main) one: the Jacobian determinant
\begin{equation}
    \det D \phi ~\equiv~ 1 ~=~ \text{constant} .
\end{equation}
The set of flow maps which satisfy this constraint is exactly the set of \emph{volume-preserving} diffeomorphisms $\{ \phi \in \diffomega ~:~  \det D \phi \equiv 1 \}$, denoted\footnote{The `S' here stands for ``special'', as in $SL(n)$, the special linear group of matrices with determinant 1.} $\sdiffomega$, or $\mathcal{S}$ for short. $\mathcal{S}$ is a subgroup of the full diffeomorphism group $\D$.

We claim without proof that essentially all of the problems that we have investigated so far in terms of densities have their counterparts in terms of flow maps, and that the structure of these problems is nearly identical. For example, reachability is an equivalence relation on the state space $\mathcal{S}$. (In the next subsection, we will show that $\mathcal{S}$ is more than just a group -- it is a manifold, and thus a \emph{Lie group}. The classes of mutually reachable states are exactly the connected components of $\mathcal{S}$. In particular, since the initial state is almost always taken to be the identity $e$, the reachable set is the \emph{identity component} of $\mathcal{S}$, denoted $\mathcal{S}_0$.)

We may also pose the optimal mixing and optimal state-transfer problems in terms of flow maps as follows.

\begin{prob}[Optimal Mixing, Flow Maps] \label{mixing_flow_maps_prob}
	Given initial and target densities $\rho_0,\rho_* \in \P$, time horizon $T$, norm $\| \cdot \| : \V \to \R$, metric $d: \P \times \P \to \R$, and trade-off parameter $\alpha >0$, solve
	\begin{equation} %
		\inf_{\phi, v} \, \int_0^T \| v_t \|^2 \, dt \, + \, \alpha d^2 \big( (\phi_T)_\# \rho_0,\rho_* \big) \qquad \text{s.t.} \qquad \begin{cases}
			\partial_t \phi = v \of \phi \\
			\nabla \cdot v = 0 \\
			\phi_0 = e
		\end{cases} \hspace{-4mm}
	\end{equation}
\end{prob}

\begin{prob}[Optimal State-Transfer, Flow Maps] \label{state_transfer_flow_maps_prob}
	Given mutually reachable initial and final states $\phi_0, \phi_1$ and a norm $\| \cdot \| : \V \to \R$, solve
	\begin{equation} %
		\inf_{\phi, v} \, \int_0^1 \| v_t \|^2 \, dt \qquad \text{s.t.} \qquad \begin{cases}
			\partial_t \phi = v \of \phi \\
			\nabla \cdot v = 0 \\
			\phi_0 ~ \text{fixed} \\
			\phi_1 ~ \text{fixed}
		\end{cases}
	\end{equation}
    Again, we refer to the value of this problem as the \emph{transfer cost} from $\phi_0$ to $\phi_1$, and denote it $\T_\mathcal{S}(\phi_0,\phi_1)$.
\end{prob}

Just as before, the square root of the transfer cost $\T_\mathcal{S}^{1/2}$ forms a metric on the reachable set, and this enables a decomposition of the optimal mixing problem. In this setting, the mixing problem decomposes into two component subproblems (rather than three): selection of an optimal final state $\phi_f$, and selection of a minimum-length path $\phi$ in the reachable set from $e$ to $\phi_f$. Unlike in the Eulerian setting, however, we do not need to further optimize over velocity fields $v$ which generate such a path: $\phi$ and $v$ are in 1-1 correspondence, and $v$ can be immediately recovered from \eqref{flow_eqn} as
\begin{equation} \label{v_recovery_eqn}
	v_t ~=~ (\partial_t \phi_t) \of \phi_t^{-1} .
\end{equation}
(Note that such a $v$ is necessarily incompressible since $\det D \phi \equiv 1$.) Our objective in this section is not to prove these facts, but instead to explore the geometry and structure of this Eulerian-Lagrangian correspondence, a task which we pick up in the next subsection.

\subsection{The Riemannian Submersion} \label{riemannian_submersion_sec}

Just as before, when the norm $\| \cdot \| = \| \cdot \|_\V$ is induced by an inner product $\lb \cdot , \cdot \rb_\V$, the reachable set $\mathcal{S}_0$ carries the formal structure of an infinite-dimensional Riemannian manifold. This is almost immediate to see in this setting since there are no equivalence classes of $v$ to worry about. Here, the manifold $M$, tangent spaces $T_xM$, and inner product $\lb \cdot , \cdot \rb_x$ are given respectively by:
\begin{enumerate}
    \item the reachable set $\mathcal{S}_0$,
    \item the sets of derivatives $\eta_0^\prime$ of curves $\eta: (-\varepsilon , \varepsilon) \to \mathcal{S}_0$ with $\eta_0 = \phi$,
    \item the inner product
    \begin{equation} \label{riemannian_metric_lagrangian_eqn}
        \lb u^\phi , v^\phi \rb_\phi ~=~ \lb u , v \rb_\V ~:=~ \lb K u , K v \rb_{L^2} .
    \end{equation}
\end{enumerate}
Note that in the above definition of the Riemannian metric, $u^\phi$ and $v^\phi$ denote tangent vectors at $\phi$, while $u$ and $v$ denote the corresponding spatial velocity fields (as related by \eqref{flow_eqn} and \eqref{v_recovery_eqn}):
\begin{equation} \label{tangent_vector_translation_eqn}
    v^\phi ~=~ v \of \phi \qquad \Leftrightarrow \qquad v ~=~ v^\phi \of \phi^{-1} .
\end{equation}
The Riemannian distance $m_\mathcal{S}$ is then seen to be equal to the square root of the transfer cost $\T_\mathcal{S}^{1/2}$ by the same arguments as in Section \ref{metrics_decomp_sec}.

As both a group and a manifold, $\mathcal{S}_0$ carries the formal structure of a \emph{Lie group}, with the incompressible vector fields on $\Omega$ as the corresponding Lie algebra. The tools of Lie theory thus apply, for example, translation of tangent vectors on $\mathcal{S}_0$ (i.e., vector fields on $\Omega$) by left- and right-action. We will say more about this in the following subsection. The important point here is that \eqref{tangent_vector_translation_eqn} describes the right-translation of tangent vectors $v^\phi \in T_\phi \mathcal{S}_0$ to spatial velocity fields $v \in T_e \mathcal{S}_0$, and so the metric \eqref{riemannian_metric_lagrangian_eqn} is \emph{right-invariant} by construction, a fact which we will remark on later.

Now, observe that given a fixed initial density $\rho_0$, the pushforward operation \eqref{pushforward_eqn} defines a map $\A = \A_{\rho_0}$ from $\mathcal{S}_0 \to \Reach$:
\begin{equation} \label{submersion_eqn}
    \A_{\rho_0} : \mathcal{S}_0 \to \Reach , ~ \phi \mapsto \phi_\# \rho_0 =: \rho .
\end{equation}
This map is onto (by the definition of reachability), but it is many-to-one in general: as established in the previous subsection, its fibers are exactly the measure-preserving relabeling classes, so that $\Reach$ is recovered as the quotient of $\mathcal{S}_0$ under relabeling. We now investigate the geometric structure of this quotient.

Since $\A$ is a map between manifolds, we may take its differential. By matching the dynamics \eqref{continuity_eqn}, \eqref{flow_eqn} under the correspondences \eqref{v_recovery_eqn}, \eqref{submersion_eqn}, we may read off the differential $D \A = D \A_{\rho_0}$ as
\begin{equation} \label{submersion_differential_eqn}
    D \A_{\rho_0} (\phi) [v^\phi] ~=~ - \nabla \cdot (\rho v) , \qquad \rho = \phi_\# \rho_0 , \quad v = v^\phi \of \phi^{-1} .
\end{equation}
By the definition \eqref{tangent_space_tau_eqn} of the tangent spaces $T_\rho \Reach$, the differential $D \A(\phi)$ is onto for each $\phi$, meaning that the map $\A$ is a \emph{submersion} of $\mathcal{S}_0$ into $\Reach$. What makes this submersion \emph{Riemannian} is that each differential $D \A(\phi)$ is, under the identification of $T_\rho \Reach$ with minimal-norm representatives (see Section \ref{riemannian_structure_sec}), exactly an orthogonal projection with respect to the inner product $\lb \cdot , \cdot \rb_\phi$. Indeed, viewing the image of $D\A(\phi)$ through this identification, the differential sends $v^\phi \mapsto \bar{v}$, which is precisely the projection $\Theta_{\rho}$ of Section \ref{riemannian_structure_sec}:
\begin{equation} \label{dA_theta_eqn}
    D \A_{\rho_0}(\phi) ~=~ \Theta_{\rho} .
\end{equation}

The submersion property here essentially holds by construction. Indeed, the metric on $(\Reach,m)$ was \emph{defined} in Section \ref{riemannian_structure_sec} as the $\V$-inner product of minimal representatives, which is exactly the statement that $D\A$ is an isometry on its image:
\begin{equation}
    \lb u , v \rb_\rho ~:=~ \lb \bar{u} , \bar{v} \rb_\V ~=~ \big\lb D\A_{\rho_0}(\phi)[u^\phi] , \, D\A_{\rho_0}(\phi)[v^\phi] \big\rb_\phi ,
\end{equation}
which equals $\lb u^\phi , v^\phi \rb_\phi$ if $u^\phi,v^\phi \in \Im(D\A_{\rho_0} (\phi))$. Following \citet{Otto2001}, we may summarize this by saying that the geometry of $(\Reach,m)$ is \emph{induced} by Riemannian submersion on $(\mathcal{S}_0,m_\mathcal{S})$ via $\A_{\rho_0}$. This is an extremely natural operation, as Riemannian submersion is the natural notion of quotient in the category of Riemannian manifolds. A cartoon depiction of the Riemannian submersion is shown in Figure \ref{fig:submersion}.

\begin{figure}[ht]
    \centering
    \includegraphics[width=0.35\linewidth]{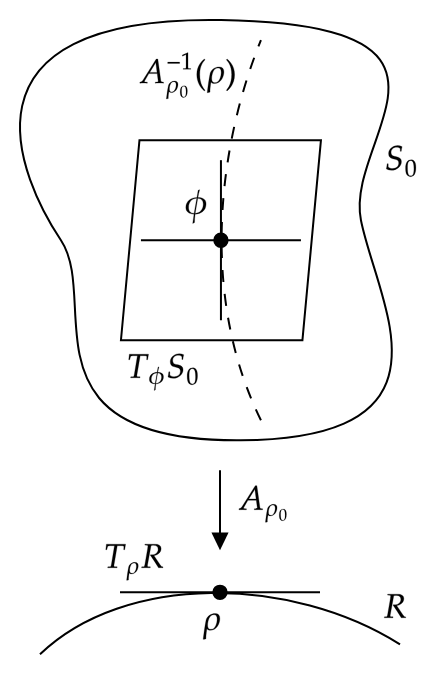}
    \caption{Depiction of the Riemannian submersion $\A_{\rho_0}$ relating the manifolds $\mathcal{S}_0$ and $\mathcal{R}$. The point $\phi$ and its fiber $\A_{\rho_0}^{-1}$ are sent to the point $\rho$. The tangent space $T_\phi \mathcal{S}_0$ is sent to the tangent space $T_\rho \mathcal{R}$ by the differential $D\A_{\rho_0}$. The vertical subspace is the nullspace of the differential, tangent to the fiber. The horizontal subspace is its orthogonal complement, isometric to $T_{\rho}\mathcal{R}$.}
    \label{fig:submersion}
\end{figure}

Since the differential $D \A(\phi)$ is an orthogonal projection, the subspaces $\Im(D\A(\phi))$ and $\Nu(D \A(\phi))$ form an orthogonal decomposition of $T_\phi \mathcal{S}_0$. In the language of differential geometry, $\Im(D\A(\phi))$ is called the \emph{horizontal} subspace and $\Nu(D\A(\phi))$ the \emph{vertical} subspace, and a tangent vector in each is called \emph{horizontal} or \emph{vertical} accordingly. The machinery of Riemannian geometry then allows us to make the following statements:
\begin{enumerate}
    \item If $\eta: [a,b] \to \mathcal{S}_0$ is a geodesic for which $\eta_t'$ is horizontal at some time $t \in [a,b]$, then $\eta_t'$ is horizontal for \emph{every} time $t \in [a,b]$. Such a geodesic is called a \emph{horizontal} geodesic.
    \item If $\eta$ is a horizontal geodesic in $(\mathcal{S}_0,m_\mathcal{S})$, then its image $\rho = \A_{\rho_0} (\eta)$ is a geodesic in $(\Reach,m)$.
\end{enumerate}
This explains the parallel between our earlier result characterizing the incompressible Euler equation as a geodesic equation in the space of densities and the famous result of \citet{Arnold1966} characterizing this same equation as a geodesic equation in the space $\sdiffomega$.

Lastly, we remark that it is exactly the right-invariance of the metric \eqref{riemannian_metric_lagrangian_eqn} which makes it possible to write these geodesic equations as PDEs governing the evolution of $v$. A priori, a geodesic depends on the metric and thus on the base points $\phi \in \mathcal{S}_0$, but right-invariance allows us to right-translate the geodesic to the identity without changing the metric, reducing the geodesic equation to an equation for the spatial velocity field $v$ alone (which lives entirely in the Lie algebra). This is an insight of Arnold's, and leads to the compact, momentum-based Euler-Poincar\'e (EPDiff) form of the geodesic equations
\begin{equation}
    \partial_t \mathfrak{m} ~=~ - \ad_v^* \mathfrak{m} ,
\end{equation}
where $K^*K$ represents the inertia operator and $\mathfrak{m} = K^*Kv$ the momentum, as alluded to in the earlier remark of Section \ref{general_k_sec}. 
We develop this Lie-theoretic viewpoint more thoroughly in the following subsection.

\subsection{A Lie-Theoretic Viewpoint} \label{lie_theory_sec}

We now present a Lie-theoretic viewpoint on the material presented in this paper. This viewpoint explains the structure which emerges in our problem, connects the Eulerian and Lagrangian settings, clarifies the analogy between classical optimal transport and the ``incompressible optimal transport'' we have presented here, and connects these theories with the field of geometric mechanics. In short, these correspondences may all be understood in terms of Lie group and Lie algebra actions on homogeneous spaces. We find it most revealing to first reconstruct the geometric theory of classical optimal transport starting from this algebraic viewpoint (which is not often emphasized in the optimal transport literature). The ``incompressible optimal transport'' we have presented in this paper can then be understood as a small alteration of this construction. We note that similar constructions have recently been developed in several other settings \citep[e.g.,][]{khesin2025universal, bauer2026wasserstein}; see Section \ref{conclusion_sec} for further discussion.

\subsubsection{Classical Optimal Transport}

We begin with an underlying physical space $\Omega$, which we assume to be a compact, connected, $n$-dimensional Riemannian manifold. We consider two fundamental objects associated with $\Omega$, the first being the diffeomorphism group of $\Omega$, denoted $\diffomega$ or $\D$, and the second being the space of probability densities on $\Omega$, denoted $\probomega$ or $\P$. We begin by describing $\D$, and then move on to describing $\P$ and how the two objects interact.

The diffeomorphism group $\D$ is the group of all diffeomorphisms $\phi$ from $\Omega \to \Omega$. $\D$ is a group under composition, inverse, and identity. Moreover, it has the formal structure of an infinite-dimensional manifold, and is thus (formally) an infinite-dimensional Lie group. The corresponding Lie algebra is given by the set of vector fields on $\Omega$, denoted $\vecomega$ or $\d$. The Lie bracket in $\d$ is given by the usual Lie bracket of vector fields, which may be written out in coordinates as
\begin{equation}
    [u,v] ~=~ (u \cdot \nabla) \, v - (v \cdot \nabla) \, u ~=~ Dv[u] - Du[v].
\end{equation}
The Lie exponential map is given by the time-1 map of the flow. That is, we define the flow as in \eqref{flow_eqn} (with $v$ constant-in-time), and then the operator $\exp: \d \to \D$ is defined by $\exp(v) = \phi_1$. The composition operation in the group and bracket operation in the algebra may be connected through Baker-Campbell-Hausdorff and related formulas.

The space of probability densities $\P$ is the space of all densities $\rho: \Omega \to \R$ which are nonnegative $\rho \geq 0$ and integrate to one $\int_\Omega \rho = 1$. In this construction, $\P$ has no algebraic structure of its own.

The group $\D$ acts on $\Omega$ through group action $(\phi,x) \mapsto \phi(x)$ and this induces group actions on the spaces of various quantities which are defined on $\Omega$ -- including, in particular, the space of probability densities $\P$. The canonical group action induced on $\P$ is\footnote{Note that the submersion $\A_{\rho_0}$ from the previous section is $\A(\cdot,\rho_0)$, the group action evaluated at $\rho_0$.}
\begin{equation}
    \A : \D \times \P \to \P , ~ (\phi,\nu) \mapsto \phi_\# \nu .
\end{equation}
Since $\Omega$ is connected, this action is transitive\footnote{This requires additional regularity assumptions on the densities in $\P$, such as being smooth and supported everywhere. This regularity is not relevant to the following discussion, however.}, i.e., for any two probability densities $\nu, \rho \in \P$, there exists some $\phi \in \D$ such that $\rho = \phi_\# \nu$. This means that $\P$ forms a \emph{homogeneous space} for $\D$. Such spaces have very special properties. For example, every homogeneous space admits a representation in terms of cosets of the group. This is explained as follows. Given a reference density $\rho_0$, we first consider the stabilizer (or isotropy) subgroup
\begin{equation}
    \D_{\rho_0} ~:=~ \{ \phi \in \D ~:~ \phi_\# \rho_0 = \rho_0 \} .
\end{equation}
Then, we may identify each point $\rho \in \P$ with a coset\footnote{Note that while $\D_{\rho_0}$ is a subgroup of $\D$, it is not \emph{normal} in $\D$. Thus, the objects $[\rho]$ are properly cosets and not subgroups.} $[\rho] \in \D / \D_{\rho_0}$. In words, each density $\rho \in \P$ is identified with the set of maps $\phi \in \D$ which push the reference $\rho_0$ forward onto $\rho$:
\begin{equation}
    [\rho] ~=~ \{ \phi \in \D ~:~ \phi_\# \rho_0 = \rho \} .
\end{equation}
This is essentially the construction of the relabeling classes of $\phi$ described in Section \ref{flow_maps_sec}. We may also use the notation $[\rho]_{\rho_0}$ if we wish to emphasize the specific reference $\rho_0$ which identifies the cosets.

Similarly, the algebra $\d$ acts on $\Omega$ through algebra action $(v,x) \mapsto v(x)$, and this induces algebra actions on the spaces of various quantities which are defined on $\Omega$, including $\P$. The canonical algebra action induced on $\P$ is
\begin{equation}
    \a : \d \times \P \to T\P , ~ (v,\nu) \mapsto - \nabla \cdot (\nu v) .
\end{equation}
Just as points in the homogeneous space $\P$ can be identified with cosets of the Lie group $\D$, tangent vectors in $T\P$ can be identified with cosets of the Lie algebra $\d$. Given a base point $\rho \in \P$, we first consider the stabilizer/isotropy subalgebra
\begin{equation}
    \d_\rho ~:=~ \{ v \in \d ~:~ -\nabla \cdot(\rho v) = 0 \} .
\end{equation}
Then, we may identify each tangent vector $\tau \in T_\rho \P$ with a coset $[\tau] \in \d / \d_\rho$. In words, each tangent vector $\tau \in T_\rho \P$ is identified with the set of vector fields $v \in \d$ which generate it:
\begin{equation}
    [\tau] ~=~ \{ v \in \d ~:~ \tau = - \nabla \cdot (\rho v) \} .
\end{equation}
This is essentially the construction of the equivalence classes of $v$ described in Section \ref{riemannian_structure_sec}. We may also use the notation $[\tau]_\rho$ if we wish to emphasize the specific base point $\rho$ which identifies the cosets.

Now, if we have a Riemannian metric $\lb \cdot , \cdot \rb_\phi$ defined on $\D$, this naturally induces a Riemannian metric $\lb \cdot , \cdot \rb_\rho$ on $\P$ via Riemannian submersion. At the base point $\phi \in \D$, the stabilizer subalgebra $\d_\rho$ (where $\rho = \phi_\# \rho_0$) forms a subspace of $T_\phi \D$. This is exactly the vertical subspace of the Riemannian submersion. The metric $\lb \cdot , \cdot \rb_\phi$ then allows us to identify an orthogonal complement $\d_\rho^\perp$, which is exactly the horizontal subspace of the submersion. Tangent vectors $\tau \in T_\rho \P$ which are identified with cosets in $\d / \d_\rho$ may then also be identified with elements in the orthogonal complement $\d_\rho^\perp$. The natural metric on $\P$ is then given by restricting the metric on $\D$ to each horizontal subspace and applying this identification.

We note that the Riemannian metric $\lb \cdot , \cdot \rb_\phi$ which is defined on $\D$ is essentially a free choice in this construction. It may or may not be induced by an underlying metric on $\Omega$, and it may or may not be flat, left- or right-invariant, or have any other number of nice properties. In classical optimal transport, for example, the Riemannian metric $\lb \cdot , \cdot \rb_\phi$ is defined by the formula
\begin{equation}
    \lb u^\phi , v^\phi \rb_\phi ~=~ \int_\Omega \lb u^\phi(x) , v^\phi(x) \rb_{\phi(x)} \, \rho_0(x) \, dx ~=~ \int_\Omega \lb u(x) , v(x) \rb_x \, \rho(x) \, dx ,
\end{equation}
where $\lb \cdot , \cdot \rb_x$ is the Riemannian metric on the underlying space $\Omega$ and $\rho = \phi_\# \rho_0$. It is neither left- nor right-invariant, but it is flat if $\Omega$ is flat, that is, if $\lb \cdot , \cdot \rb_x$ is independent of $x$ (and thus $\lb \cdot , \cdot \rb_\phi$ is independent of $\phi$). Interestingly, the geodesic equation in this setting -- given by the pressureless Euler equation \eqref{pressureless_euler_eqn} -- also reduces to a PDE in $v$ alone, but for a totally different reason than in the case described in the previous subsection. Both reductions rely on an identification of tangent spaces in order to represent tangent vectors in the Lie algebra, but while this identification is carried out by right-translation in our setting, it is carried out by parallel transport in the setting of classical optimal transport.

\subsubsection{Incompressible Optimal Transport}

The algebraic and geometric construction of our ``incompressible optimal transport'' is nearly identical to that of classical optimal transport, but with two small alterations.

First, instead of considering the action of the full diffeomorphism group $\D$ on $\P$, we consider just the action of the subgroup $\mathcal{S}$ of volume-preserving diffeomorphisms. $\mathcal{S}$ is a Lie subgroup of $\D$, and its Lie algebra $\s$ -- the set of divergence-free vector fields -- is a Lie subalgebra of $\d$. $\mathcal{S}$ and $\s$ inherit the Lie bracket and Lie exponential from $\D$ and $\d$. Note that $\s$ is precisely the constraint set $\Nu(\nabla \cdot \ph) = \Im (\Pi_{\1,K}) \subset \V$ appearing throughout Sections \ref{problem_structure_sec}-\ref{gradient_flows_sec}; the incompressibility constraint there is the statement that admissible velocity fields lie in this subalgebra.

The group and algebra actions of $\mathcal{S}$ and $\s$ on $\P$ are identical to those of $\D$ and $\d$. There is one small complication here, in that the action of $\mathcal{S}$ on $\P$ is no longer transitive, and so $\P$ does not form a homogeneous space for $\mathcal{S}$. This is exactly the problem of reachability discussed in Section \ref{reachability_sec}. However, this may be remedied by restricting $\P$ to the reachable set $\Reach = \Reach_{\rho_0}$, i.e., the orbit of $\rho_0$ under $\mathcal{S}$. $\mathcal{S}$ acts transitively on any orbit, and the construction of stabilizer subgroups/subalgebras and cosets can then proceed unimpeded.

The second alteration is that we place a different metric on $\mathcal{S}$:
\begin{equation}
    \lb u^\phi , v^\phi \rb_\phi ~=~ \lb u , v \rb_\V ~:=~ \lb K u , K v \rb_{L^2} .
\end{equation}
In contrast to classical optimal transport, this metric is not always induced by an underlying metric on $\Omega$. For example, when $K=I$, it is, but when $K = D$, it is not. In this latter case, the metric $\lb \cdot , \cdot \rb_\phi$ depends not on how particles are moving through space, but on how they are moving \emph{relative to each other}. This metric also has different properties: it is not flat (or left-invariant), but it is right-invariant. As discussed in the previous subsection, this has consequences for how the reduction of the geodesic equations to a PDE for $v$ may be carried out. Again, in our setting this occurs by right-translation rather than by parallel transport.

\section{Discussion} \label{conclusion_sec}

In this paper, we formulated a version of the fluid mixing problem as an optimal control problem which trades off between fluid mixedness and mixing effort, connecting the classical theories of fluid mixing and optimal mass transport. We showed that the set of reachable fluid configurations can formally be endowed with a Riemannian metric induced by the control effort, that the mixing problem decomposes into a series of more fundamental geometric problems, and that flow fields which are maximally efficient at mixing must be geodesics in this Riemannian space. Using the Riemannian structure, we derived necessary conditions for optimality, geodesic equations, and gradient flows, and we introduced a generalization of the Helmholtz decomposition to help interpret them. We treated in detail both the energy-penalized ($L^2$) and enstrophy-penalized ($\dot{H}^1$) cases. The energy-penalized case recovers several classical results, including the incompressible Euler equation and its Clebsch representation, but appears to be ill-posed. This is precisely the kind of information the framework-first approach is meant to provide: it indicates that the kinetic energy is an unnatural penalty for mixing. The enstrophy-penalized case is expected to be much better behaved, and we provided numerical simulations of both geodesics and gradient flows in this setting. Finally, we showed how the Eulerian-Lagrangian correspondence and the construction of generalized transport metrics can be understood in terms of Lie group and Lie algebra actions on homogeneous spaces and Riemannian submersion.

The connection between incompressible fluid mixing and optimal transport runs in both directions. On one hand, the fluid mixing problem is essentially an optimal transport problem and shares much of its formal structure, although the transport occurs across \emph{scales}\footnote{This idea can be formalized in the frequency-domain picture as the transport of the ``mass'' in the Fourier transform to higher wave numbers -- consistent with the definition of mixedness via weak metrics such as negative-order Sobolev.} rather than across space. On the other hand, incompressible fluid mixing provides a case study of a ``generalized optimal transport'' problem. A natural question to ask of such problems is when they retain the structure of classical optimal transport -- in particular, its Riemannian structure, which provides the apparatus to do calculus on these spaces. This case study seems to provide a clue.

The Lie-algebraic construction of generalized transport metrics is remarkably fundamental, requiring only a Lie group or Lie algebra carrying a geometry and acting on a homogeneous space. Beyond the examples presented here, this structure has been identified in linear/Gaussian (Bures-Wasserstein) optimal transport \citep{malago2018wasserstein}, unbalanced optimal transport of densities \citep{gallouet2018camassa} and of Riemannian metrics \citep{bauer2026wasserstein}, and vector/matrix-valued optimal transport \citep{khesin2025universal}, and the same construction organizes the geometric-hydrodynamical treatment of many physical equations \citep{khesin2021geometric}. We expect that it underlies many other examples as well. This abundance motivates both the abstract study of generalized optimal transport metrics in their own right and their connections to geometric mechanics and control.

Many questions remain open. We have said little about the precise characterization of reachable sets, the final state selection problem, or how to compute trajectories achieving state transfer. The formal picture presented here also deserves rigorous analysis: while there is a substantial literature on the Lagrangian setting with the $L^2$ geometry, both the Eulerian setting and the $\dot{H}^1$ geometry remain comparatively unexplored. Finally, the connections between our framework and existing results on mixing rates, as well as the use of gradient flows for mixing, are natural directions for future work.

\backmatter

\bmhead{Acknowledgements}

Thank you to Igor Mezi\'{c}, Matt Jacobs, Katy Craig, Alexandre Leonelli, Jethro Warnett, and the attendees at the 2025 Optimal Transport Summer School at UC Santa Barbara for valuable discussions during the development of this work. We also thank the anonymous reviewers of the conference version of this work, whose careful reading identified an error in one of our proofs (the corrected proof appears here).

\section*{Declarations}

\begin{itemize}
\item Funding: This work was supported in part by NSF ECCS-2453491 and AFOSR FA9550-23-F-0014.
\item Conflict of interest/Competing interests: The authors have no competing interests to declare.
\item Ethics approval and consent to participate: Not Applicable.
\item Consent for publication: All authors consent to publication.
\item Data availability: No external datasets were analyzed in this study. All numerical results presented in this paper were generated by the simulation code (see below).
\item Materials availability: Not Applicable.
\item Code availability: The code used to generate the numerical results in this paper is available at \url{https://github.com/johnigraszekcode/Fluid-Mixing-Code}.
\item Author contribution: B.B. proposed the problem. M.E. developed the theoretical framework, wrote the proofs, and drafted the manuscript, with ongoing input from B.B. J.I. implemented the numerical simulations, produced the figures, and wrote the corresponding sections. All authors reviewed, edited, and approved the final manuscript.
\end{itemize}

\begin{appendices}

\end{appendices}

\bibliography{library}%

\end{document}